\documentclass{CSML}
\pdfoutput=1

\usepackage{lastpage}
\newcommand{\dOi}{14(2:5)2018}
\lmcsheading{}{1--\pageref{LastPage}}{}{}%
{Oct.~27, 2017}{Apr.~25, 2018}{}

\usepackage{amssymb}
\usepackage{MnSymbol}
\usepackage{varioref}
\usepackage[all]{xy}
\usepackage{yfonts}
\usepackage{url}
\usepackage{graphicx}
\usepackage{color}
\usepackage{xcolor}
\usepackage{microtype}
\usepackage{hyperref}
\hypersetup{hidelinks}

\newcommand{\e}{\varepsilon} 			
\newcommand{\N}{\mathbb N}
\newcommand{\G}{\mathbb G} 			
\newcommand{\C}{\mathbb C} 			
\newcommand{\B}{\mathbb B} 			
\newcommand{\I}{\mathbb I} 					
\newcommand{\DA}{{\textswab 2}} 			
\newcommand{\FD}{\mathcal D}   		
\newcommand{\FE}{\mathcal E}   		
\newcommand{\FF}{\kappa_\ast}		
\newcommand{\FU}{\mathcal U}		
\newcommand{\pre}[1]{Set^{{#1}}}   

\newcommand{\Mapping}[3]{\xymatrix{{#1} \ar[r]^{{#2}} & {#3}}}
\newcommand{\IsoMapping}[3]{\xymatrix{{#1} \ar@{ >->>}[r]^{{#2}} & {#3}}}
\newcommand{\GO}[1]{\G\hspace{-0.3ex}\downarrow\hspace{-0.3ex}{#1}} 
\newcommand{\CO}[1]{\C\hspace{-0.3ex}\downarrow\hspace{-0.3ex}{#1}} 
\newcommand{\ol}[1]{\overline{#1}}
\newcommand{\myref}[1]{(\ref{#1})}

\newcommand{\txtnode}[2]{\fbox{\txt<#1>{#2}}}
\newcommand{\cb}[2]{#1{:}\,#2}
\newcommand{\Double}[4]{
\xymatrix{#1 \ar[r]_{#2} \ar@<1.ex>[r]^{#3} & #4}
}
\newcommand{\DoubleWithCocone}[7]{
\xymatrix{#1\ar@/_1.2pc/[rrr]_{#7} \ar[rr]_(0.65){#2} \ar@<1.ex>[rr]^(0.65){#3}
&& #4\ar[r]^{#5} & #6} }

\newcommand\ingr[2]{\includegraphics[height = #1in]{Images/#2}}

\newdir{ (}{{}*!/-5pt/@^{(}}    
\newdir{ )}{{}*!/-5pt/@_{(}}    
\newdir{ >}{{}*!/-5pt/@{>}}     

\begin{document} 

\title[Van Kampen Colimits and Path Uniqueness]{Van Kampen Colimits and Path Uniqueness\rsuper*}

\titlecomment{{\lsuper*}This paper is an extended version of the CALCO '17 paper ''{B}eing {V}an {K}ampen is a {U}niqueness {P}roperty in {P}resheaf {T}opoi'' \cite{kw17Calco}}

\author[H.\ K{\"o}nig]{Harald K{\"o}nig}	
\address{Department of Informatics, University of Applied Sciences FHDW Hannover, Freundallee 15, 30173 Hannover, Germany}	
\email{harald.koenig@fhdw.de}  

\author[U.\ Wolter]{Uwe Wolter}	
\address{Department of Informatics, University of Bergen, P.O.Box 7803, 5020 Bergen,  Norway}	
\email{Uwe.Wolter@uib.no}  

\begin{abstract}
Fibred semantics is the foundation of the model-instance pattern of software engineering. Software models can often be formalized as objects of {\em presheaf topoi}, i.e, categories of objects that can be represented as algebras as well as coalgebras, e.g., the category of directed graphs. Multimodeling requires to construct {\em colimits} of models, decomposition is given by {\em pullback}. Compositionality requires an exact interplay of these operations, i.e., diagrams must enjoy the {\em Van Kampen} property. However, checking the validity of the Van Kampen property algorithmically based on its definition is often impossible.
\par
In this paper we state a necessary and sufficient yet efficiently checkable condition for the Van Kampen property to hold in presheaf topoi. It is based on a uniqueness property of path-like structures within the defining congruence classes that make up the colimiting cocone of the models. We thus add to the statement ''Being Van Kampen is a Universal Property'' by Heindel and Soboci\'{n}ski the fact that the Van Kampen property reveals a presheaf-based structural {\em uniqueness} feature.
\end{abstract} 


\maketitle
\section{Introduction}
A presheaf topos is a category, that is based on an algebraic signature with unary operation symbols. Presheaves can also be considered as intersection of algebras and coalgebras \cite{Kahl2014}. Van Kampen Colimits are a generalization of Van Kampen squares \cite{Sobo2004}. In \cite{WK2013} we gave a necessary and sufficient condition for a pushout to be a Van Kampen square in a presheaf topos. In the present paper a corresponding criterion is given for all colimiting cocones.
\subsection{Motivation}
Software engineering and especially model-driven software development requires the decomposition of large models into smaller components, i.e., successful development of large applications requires system design fragmentation. Vice versa, a comprehensive viewpoint of a related ensemble of heterogenous software-engineering components is taken up by considering the {\em amalgamation} (union) of these artefacts modulo their relations amongst each other. This assembly shall not only be carried out on a syntactical level (models), but in the same way on the semantical level (instances). This interplay between assembly and disassembly shows that composition and {\em correct} decomposition of an instance of a model into instances of the model components always accompany each other. It can be shown that correctness, i.e., {\em compositionality} \cite{EGW1998,Fiadeiro2005} is not always guaranteed \cite{DBLP:books/daglib/0028275}.
\par
{\em Fibred semantics} adheres to the model-instance pattern, a standard viewpoint in software engineering: A model $M$ is an object of an appropriate category $\C$, semantics is given by the comma category $\CO{M}$. In each object $\tau\in \CO{M}$, $\tau:I\to M$, $I$ is the instance structure and $\tau$ is its typing. Amalgamation is colimit (of the arrangement of components) and decomposition is performed by taking pullbacks along the cocone morphisms of the colimit.
\par
To wit: Compositionality means that colimit of semantics (instances) is controlled by colimit of syntax (models) such that pullback of the instance colimit retrieves the original instances. Thus compositionality is equivalent to the {\em Van Kampen property} \cite{VKAsBicolimits}, an abstract characteristic which determines an exactness level for the interaction of colimits and pullbacks. It is thus often necessary to check validity of this property. However, since the definition of the property comes in terms of an equivalence of categories, see Def.\ref{def:VKCocone} in the present paper, algorithmic verification based on the definition is hard even for a finite number of finite models, because the involved comma categories are infinite nevertheless.
\par
Artefacts like UML- or ER-models are based on directed multigraphs, which in turn can be coded as a functor category $Set^\B$, where $\B$ has objects $E$ (edges) and $V$ (vertices) and non-identical arrows $s,t:E\to V$. More general metamodels, however, use more sophisticated categories $\B$, such as E-graphs for attributed graphs \cite{Ehrig2006-foagt}, bipartite graphs for Petri nets \cite{Ehrig2006-foagt}, or more complex structures for generalized sketches \cite{WD2008_a}. Hence, $Set^\B$ with $\B$ an arbitrary small category, will be the underlying category for the forthcoming investigations.
\par
Constructing colimits in a category $\C$ is an operation on {\em diagrams}, which are usually coded as functors from a small schema category $\I$ to $\C$. In order to make our results usable for software engineering, we use the older definition for diagrams: Instead of a small category, the schema $\I$ is a finite multigraph and a diagram is a graph morphism from $\I$ to $\C$ \cite{MacLane-1998}\footnote{
More precisely to the underlying graph of $\C$, see Sect.\ref{sect:prelim}}. The practical construction of colimits relies on {\em mapping paths}, i.e., chains of pairs of elements that are mapped to each other by the morphisms in the diagram, cf. Def.\ref{def:Weak-Mapping-Path} in Sect.\ref{sec:main}. Thus, colimit computation can easily be carried out algorithmically, if the diagram is finite and consists of finite artefacts.
\par
Summary: While colimit construction is easy, compositionality check (validation of the Van Kampen property) is hard. The first {\em contribution of the present paper} is a theorem (Theorem \ref{theo:main} in Sect.\ref{sec:main}), which states that a colimit in a presheaf topos has the Van Kampen property if and only if there are no ambiguous mapping paths between any pair of elements of the coproduct of the model artefacts. Thus the implementation of the colimit operation on the model level already provides the material for more efficient compositionality checking.
\par 
The second contribution is a practical algorithm, which efficiently verifies whether {\em compositionality} holds, i.e.\ whether the Van Kampen (henceforth often abbreviated ''VK'') property is satisfied. We provide a technique which combines (1) a more efficient colimit computation and (2) a simultaneous check of the VK-property.
\par
The paper is organized as follows: Sect.~\ref{sect:prelim} introduces notation and background information, Sect.~\ref{sec:main} presents the main theorem and applies it to a Software Engineering problem. Sect.~\ref{sec:to-colimits} sketches the proof idea and thus provides insight in the corresponding underlying foundations: We use a former result, in which a necessary and sufficient criterion is given for pushouts \cite{WK2013}. This result is translated to coequalizers and, finally, lifted to colimits of arbitrary diagrams. Sect.~\ref{sec:counterex} exemplifies the limitation of Theorem~\ref{theo:main}: It turns out that it is crucial to claim that the underlying category is a presheaf topos and the examples show that this is the best one can achieve. The announced practical guidelines are contained in Sect.~\ref{sec:practice}. Sect.~\ref{sec:conc} concludes with a short discussion of future research directions. 
\par 
Proofs in the  present paper are only sketched. A reader who is interested in the corresponding detailed proofs is referred to the technical report \cite{kw17}.  

\subsection{Related Work}
The Van Kampen property has its origin in algebraic topology: Topological spaces $X$ can be investigated by a covering family of $X$ which are related by their inclusions. Topological properties are expressed with the help of the fundamental groupoid. The Van Kampen Theorem \cite{MayAlgTop} states that the colimit of the fundamental groupoids of all covering spaces is the fundamental groupoid of $X$, thus inferring global properties from local ones. The original idea was stated by Seifert \cite{VKOriginal} for pushouts and was further elaborated by Van Kampen \cite{VKOriginal2}. 
\par
Inferring global properties from local ones is the heart of sheaf theory \cite{MacLane-1992}. The fibred view on sheaves is discussed in \cite{Vis15}. The application of Van Kampen's ideas to graphical modeling and to Software Engineering was invented in \cite{LS04,Sobo2004} and then further detailed in \cite{Ehrig2006-foagt} for the theory of Graph Transformations. That extensive categories and especially topoi are a reasonable playground for these theories is shown in \cite{BL2003,LS06}.
\par
{\em Amalgamation} is a requirement for a collection of artefacts in computer science \cite{EGW1998,Fiadeiro2005} which has been connected to the Van Kampen property in \cite{WK2013}. The same property is called {\em exactness} in institution theory \cite{DBLP:books/daglib/0028275}. Being ''non-exact'' seems to be a typical deficit of (1-, 2-, 3-\ldots)categories, because Lurie shows that Higher Van Kampen theorems hold in greater generality in $\infty$-topoi \cite{Lurie}.  
\par 
The importance of finding a feasible condition to check the Van Kampen property was caused by investigations of new methods in Graph Transformations \cite{DBLP:conf/gg/KonigLSW14,L-10}. That the Van Kampen property can be characterized as a bicolimit in a comprising span bicategory \cite{VKAsBicolimits} is a fundamental statement. Moreover, the Van Kampen property has been investigated in more special contexts \cite{DBLP:journals/jlp/Kahl11} and can also be described with the help of weak 2-limits in $CAT$ (https://ncatlab.org/nlab/show/van+Kampen+colimit). However, all these characterisations can hardly be applied in practice. 
\par 
The present paper is an extended version of the CALCO '17 paper ''Being Van Kampen is a Uniqueness Property in Presheaf Topoi''. In addition to the conference version, we added examples showing the limitations of our results in Sect.~\ref{sec:counterex} and we bridge the gap to concretely applicable algorithms (in Sect.~\ref{sec:practice}) of the theoretical results of the paper. 

\xymatrixrowsep{1.5pc} 
\xymatrixcolsep{1.5pc}
\section{Preliminaries}\label{sect:prelim}
This chapter recapitulates the most important notation for the following elaboration. For any category $\C$, $X\in \C$ means that $X$ is contained in the collection of objects in $\C$. A {\em diagram} in $\C$ is based on a directed multigraph $\I$, the schema for the diagram. We write $\I_0$ and $\I_1$ for the sets of vertices and edges of $\I$. Formally, a diagram $\FD:\I\to \FU(\C)$ is a graph morphism where $\FU$ denotes the forgetful functor assigning to each category its underlying graph. For convenience reasons, however, the forgetful functor will be omitted, i.e., diagrams will be denoted $\FD:\I\to \C$. This definition is used instead of the one, where $\I$ is a schema {\em category} rather than a graph, because it will turn out, that the results in this paper can easier be stated. The notions of (co-)cones and (co-)limits is the same modulo the adjunction ${\mathcal F}\dashv \FU$ where ${\mathcal F}:$ {\sc Graphs} $\to$ {\sc Cat} assigns to any graph its freely generated category, see \cite{MacLane-1998}, III, 4 for more details. Another advantage of this definition occurs in software engineering: Although the schema graph is finite, ${\mathcal F}(\I)$ may have infinitely many arrows.
\par
Vertices of $\I$ play the role of indices for diagram objects, hence, we use letters $i,j,\ldots$ for vertices. Edges of $\I$ will be depicted $\Mapping{i}{d}{j}$ and we write $i = s(d)$, $j = t(d)$ (source and target of $d$). Images of edges under a diagram $\FD: \I\to \C$ will be denoted $\Mapping{\FD_i}{\FD_d}{\FD_j}$ (slightly deviating from the usual notation $\FD(i), \FD(d)$, etc).
\par
Let $\FE, \FD: \I \to\C$ be two diagrams, then a family
\[\tau = (\tau_i: \FE_i\to \FD_i)_{i\in \I_0}\]
of $\C$-morphisms with $\tau_j\circ \FE_d = \FD_d\circ \tau_i$ for all edges $\Mapping{i}{d}{j}$ in $\I_1$ will be called a {\em natural transformation} between the diagrams and will be denoted in the usual way $\tau:\FE\Rightarrow \FD$. For any $S\in \C$, $\Delta S: \I\to \C$ denotes the constant diagram, which sends each edge of $\I$ to $id_S$. $S$ (as $\C$-object) and $\Delta S$ (as diagram) will be used synonymously. Diagrams together with natural transformations constitute the category $\C^\I$. Note that $\Delta: \C\to \C^\I$ is itself a functor, assigning to each object of $\C$ its constant diagram and to an arrow $f:A\to B$ the ''constant'' natural transformation $(f)_{i\in\I_0}$.
\par
We assume all categories under consideration to have colimits. The coproduct cocone of a family $(\FD_i)_{i\in I}$ of $\C$-objects will be denoted
\[(\Mapping{\FD_i}{\subseteq_i}{\coprod_{j\in I}\FD_j})_{i\in I}.\]
The morphisms $\subseteq_i$ are called coproduct injections. For a family of arrows $(f_i:\FD_i\to A)_{i\in I}$ we write $\vec{f}:\coprod_{i\in I}\FD_i \to A$ for the resulting unique mediating arrow.
\par
We assume all categories under consideration to have pullbacks. In the sequel, we will work with {\em chosen pullbacks}, i.e., for each pair of $\C$-arrows $B\stackrel{h}{\to} A \stackrel{k}{\leftarrow} X$ a choice
\[
\xymatrix{
Y\ar[r]^{h'}\ar[d]_{h^\ast(k)} & X\ar[d]^{k} \\
B\ar[r]_h & A
}
\]
of pullback span $(h^\ast(k), h')$ is determined once and for all. For all $h:B\to A$, $h^\ast(id_A)$ shall be chosen to be $id_B$. Whenever we deviate from these choices, this will be emphasized. It is well-known \cite{Goldblatt1984-t} that for fixed $h:B\to A$ chosen pullbacks along $h$ give rise to a (pullback) functor $h^\ast:\CO{A}\to \CO{B}$ between comma categories. Pullbacks can be composed, i.e., if $C\stackrel{h_2}{\to} B \stackrel{h_1}{\to} A$, then $h_2^\ast\circ h_1^\ast$ yields a pullback along $h_1\circ h_2$, and decomposed, i.e., if $h_1^\ast(k)$ and $(h_1\circ h_2)^\ast(k)$ are computed, the resulting universal arrow from the latter into the former pullback yields a pullback of $h_2$ and $h_1^\ast(k)$. Note, that in both cases the automatically appearing pullbacks need not be chosen.
\par
The underlying category for all further considerations is a category of {\em presheaves}, i.e., the category $\G:=Set^\B$ (with $\B$ a small (base) category, $Set$ the category of sets and mappings) of covariant functors from $\B$ to $Set$ together with natural transformations between them\footnote{
Normally presheaves are categories $Set^{\B^{op}}$, i.e., contravariant $Set$-valued functors. But we prefer the slightly deviating definition, because we found the contravariant version counterintuitive for our work. Clearly, it is easy to switch to the contravariant setting, if one inverts all arrows of $\B$.}. We will also use the term ''sort'' for the objects in $\B$ and the term ''operation (symbol)'' for the morphisms in $\B$. It is folklore that $\G$ has all  colimits and all pullbacks, which are computed sortwise, resp. $\G$ is a topos, i.e.\ a category with finite limits and colimits, which has exponents and where the subobject functor is representable \cite{Goldblatt1984-t}. $\G$ will thus also be called a {\em presheaf topos}. E.g., the category of multigraphs is a presheaf topos with $\B = (\Double{E}{t}{s}{V})$ (plus identities). The simplest presheaf topos is $Set$ ($\B = 1$, the one-object-one-morphism category).
\par
In this paper, we will make frequent use of (sortwise) coproducts, i.e., disjoint unions of sets. In order to make argumentations simpler, we will assume that for each $X\in\B$ the artefacts $(\FD_i(X))_{i\in\I_0}$ are a priori disjoint, i.e., the coproduct is obtained by simple union.
\par
An important property of presheaf topoi is (infinite) {\em extensivity}, i.e., the functor
\begin{equation}\label{eqn0.5}
\coprod: \prod_{i\in I}\GO{D_i} \to \GO{\coprod_{i\in I}D_i}
\end{equation}
assigning to each object $(f_i:A_i\to D_i)_{i\in I}$ in $\prod_{i\in I}\GO{D_i}$ the object $\coprod_{i\in I}f_{i}:\coprod_{i\in I}A_i\to\coprod_{i\in I}D_i$ in $\GO{\coprod_{i\in I}D_i}$, is an equivalence of categories for each index set $I$ and each $I$-indexed family $(D_i)_{i\in I}$ of objects in $\G$. Its ''inverse'' arises from constructing pullbacks along coproduct injections. From these facts one derives the stability of coproducts under pullbacks, i.e., if
\begin{equation}\label{eqn0.7}
\xymatrix{A_i\ar[r]^{a_i}\ar[d]_{f_i} & A\ar[d]^g
  &&& \coprod_{i\in I}A_i\ar[rr]^{\vec{a}}\ar[d]_{\coprod_{i\in I}f_i} && A\ar[d]^g\\
  				M_i\ar[r]^{g_i}				& M
  &&& \coprod_{i\in I}M_i\ar[rr]^{\vec{g}}				&& M				
  }
\end{equation}
are commutative diagrams, then the squares on the left-hand side are pullbacks for all $i\in I$, if and only if the square on the right-hand side is a pullback \cite{Goldblatt1984-t}. In general topoi, all these statements still hold for finite index sets $I$ (finite extensivity).


\section{An Equivalent Condition for the Van Kampen Property}\label{sec:main}
In this chapter we introduce the Van Kampen property and state the main result of this paper, a necessary and sufficient condition for the Van Kampen property to hold in $\G = Set^\B$.

\subsection{Van Kampen Colimits}\label{sec:VKColimits}
A commutative cocone out of a diagram $\FD:\I\to\G$ is a natural transformation
\begin{equation}\label{eqn:cocone}
\kappa:\FD\Rightarrow \Delta S.
\end{equation}
For fixed $\Mapping{i}{d}{j}$ of $\I_1$, pulling back a $\G$-arrow $\Mapping{K}{\sigma}{S}$ along $\kappa_i$ and $\kappa_j$ yields
\begin{equation}\label{eqn:pb-pointwise}
\xymatrix{
   \FE_{i} \ar@{-->}[rr]^{\FE_d} \ar@/^1.5pc/[rrrr]^(.2){\kappa'_i}
   \ar[d]_{\kappa_i^\ast(\sigma)}
&& \FE_j\ar[d]^{\kappa_j^\ast(\sigma)} \ar[rr]^{\kappa'_j}
&& K \ar[d]^\sigma
\\
   \FD_{i} \ar[rr]^{\FD_d} \ar@/_1.5pc/[rrrr]_(.2){\kappa_i}
&& \FD_j\ar[rr]^{\kappa_j}
&& S
}
\end{equation}
where the right and the outer rectangles are chosen pullbacks, $\FE_d$ is the unique completion into the right pullback, and the resulting left square is a pullback by the pullback decomposition property. The left square may, however, not be a chosen one, but it results in diagram $\FE$ as well as natural transformation $\kappa^\ast(\sigma):=(\kappa_i^\ast(\sigma))_{i\in \I_0}:\FE\Rightarrow \FD$, whose naturality squares are pullbacks. This fact gives rise to the following definition:
\begin{defi}[Cartesian Transformation] A natural transformation $\tau: \FE\Rightarrow \FD: \I\to \G$ is called {\em
cartesian} if all naturality squares are pullbacks.
\end{defi}
For a fixed diagram $\FD:\I\to \G$ let $\G^\I\Downarrow\FD$ be the full subcategory of $\G^\I\downarrow\FD$ of {\em cartesian} natural transformations. Thus, by \myref{eqn:pb-pointwise}, $\kappa^\ast$ maps objects of $\GO{S}$ to objects in $\G^\I\Downarrow\FD$. Moreover, any arrow $\gamma: \sigma\to \sigma'$ of $\GO{S}$ yields a family of arrows $(\kappa_i^\ast(\gamma))$ (universal arrows into pullbacks) of which it can easily be shown that together they yield a cartesian natural transformation $\kappa^\ast(\gamma):\kappa^\ast(\sigma)\to \kappa^\ast(\sigma')$. Thus $\kappa^\ast$ becomes a functor
\begin{equation}\label{equ:pulling-back-functor}
  \kappa^\ast: \GO{S}\to \G^\I\Downarrow\FD.
\end{equation}
\begin{defi}[Van Kampen Cocone, \cite{VKAsBicolimits}]\label{def:VKCocone}
Let $\FD:\I\to\G$ be a diagram and $\kappa:\FD\Rightarrow\Delta S$ be a commutative cocone. Then $\kappa$ has the {\em Van Kampen (VK) Property} (''$\kappa$ is VK'') if functor $\kappa^\ast$ is an equivalence of categories.
\end{defi}
As usual, a {\em colimit} (or {\em colimiting cocone}) is a universal cocone $\kappa: \FD\Rightarrow \Delta S$, i.e., for each $T\in \G$ and commutative cocone $\rho:\FD\Rightarrow \Delta T$, there is a unique $\G$-morphism $\Mapping{S}{u}{T}$ such that $\Delta u\circ \kappa = \rho$, i.e., $u\circ \kappa_i = \rho_i$ for all $i\in \I_0$. $S$ is called the {\em colimit object}.
\par
$\kappa^\ast$ has a left-adjoint $\kappa_\ast: \G^\I\Downarrow \FD \to \GO{S}$ which assigns to a cartesian natural transformation $\tau:\FE\Rightarrow \FD$ the unique arrow to $S$ out of the colimit object of the colimiting cocone of $\FE$ \cite{Sobo2004}. I.e., $\kappa_\ast$ is the (pseudo-)inverse of $\kappa^\ast$, if the VK property holds. In this case, unit and counit of the adjunction are isomorphisms. Note also that each VK cocone $\FD\Rightarrow \Delta S$ is automatically a colimit (apply $\kappa^\ast$ to $id_{\Delta S}$ and use the definition of $\FF$) such that we can use the terms ''Van Kampen cocone'' and ''Van Kampen colimit'' synonymously.
\par
Whereas the counit of this adjunction is always an isomorphism, if pullback functors have right-adjoints (and thus preserve colimits), which is true in every (presheaf) topos \cite{Goldblatt1984-t}, the situation is more involved concerning the unit of the adjunction: The easiest example of the VK property arises for the empty diagram. In this case the property translates to the fact, that the {\em initial object} $0$ is strict, i.e., each arrow $\Mapping{A}{}{0}$ is an isomorphism. This is true in all topoi \cite{Goldblatt1984-t}. In the same way, since all presheaf topoi are extensive (cf. Sect.\ref{sect:prelim}), coproducts have the Van Kampen property. But the unit fails to be an isomorphism for pushouts and coequalizers: Even in $Set$ there are easy examples of pushouts which violate the VK property \cite{Sobo2004}. In {\em adhesive categories} (and thus in all topoi \cite{LS06}) pushouts are VK, if one leg is monic, by definition. Vice versa, there are also pushouts with both legs non-monic, which enjoy this property nevertheless \cite{WK2013}. Astonishingly, coequalizers seldom are VK: Consider the shape graph $\DA := \Double{1}{d}{d'}{2}$ and the diagram $\FD:\DA\to Set$ with $\FD_1 = \{\ast_1\}, \FD_2 = \{\ast_2\}$. Clearly,
\begin{equation}\label{eqn:double-id-coeq}
\Double{\FD_1}{}{}{\FD_2}\Mapping{}{}{\{\ast\}}
\end{equation}
is a coequalizer in $Set$. Then the cartesian transformation
\[\tau:(\Double{\FE_1 := \{a,b\}}{k}{id}{\FE_2 :=\{a,b\}})\Rightarrow \FD,\]
with $k$ the non-identical bijection of $\{a,b\}$, is mapped to $id_{\{\ast\}}$ by $\FF$, i.e., $\tau \not\cong (\kappa^\ast\circ \kappa_\ast)(\tau)$.

\subsection{Equivalent Condition}
As mentioned in the introduction it is important for several software engineering scenarios to
find an easily checkable criterion for the Van Kampen property. The presented condition of this paper comes in terms of the mapping behavior of all morphisms $\FD_d$ in the diagram.
\begin{defi}[Mapping Path]\label{def:Weak-Mapping-Path}
Let $\G = \pre{\B}$ be a presheaf topos and $\FD:\I\to\G$ be a diagram w.r.t.\ shape graph $\I$. Let $\I_1^{op} := \{d^{op}\mid d\in \I_1\}$.
\begin{itemize}
  \item A {\em Path Segment} of sort $X\in\B$ is a triple $(y,\delta, y')$ with $\delta\in \I_1\cup\I_1^{op}$ and\footnote{
    Whenever $\Mapping{i}{d}{j}\in \I_1$ and we apply a mapping in the family $((\FD_d)_X: \FD_i(X)\to \FD_j(X))_{X\in\B}$, we write $\FD_d$ instead of $(\FD_d)_X$.}
  \[\begin{array}{lcl}
    \mbox{If }\delta = d\in \I_1 & \mbox{ then } & y\in \FD_{s(d)}(X), y' = \FD_d(y)\in \FD_{t(d)}(X) \\
    \mbox{If } \delta = d^{op}\in \I_1^{op} & \mbox{ then } & y'\in \FD_{s(d)}(X), y = \FD_d(y')\in \FD_{t(d)}(X)
  \end{array}\]
  Two path segments $(y_1, \delta_1, y_1')$ and $(y_2, \delta_2, y_2')$ of sort $X$ are equal if $y_1 = y_2$, $\delta_1 = \delta_2$, and $y'_1 = y'_2$. Moreover, two path segments are {\em weakly equal},  $(y_1, \delta_1, y_1')=_w(y_2, \delta_2, y_2')$ in symbols, if $(y_1, \delta_1, y_1')=(y_2, \delta_2, y_2')$ or $(y_1, \delta_1, y_1')=(y_2', \delta_2^{op}, y_2)$.\footnote{$(d^{op})^{op}:=d$.}
  \item A {\em Non-empty Mapping Path in $\FD$} of sort $X\in\B$ is a sequence
  \[P=[(y_0, \delta_0, y_1), (y_1, \delta_1, y_2), (y_2, \delta_2, y_3), \ldots, (y_{n-1},
  \delta_{n-1}, y_n)]\]
  of path segments of sort $X$, where any third component of a segment coincides with the first component of its successor segment\footnote{
  By the introductory remarks on disjointness of artefacts, this means that the third component and the successor's first component are elements of the same $\FD_i$.}, and where $n\ge 1$.
  We say that the above path connects $y_0$ with $y_n$ in $\FD$.
  \item For each $y\in \FD_i(X)$, where $i\in\I_0$ and $X\in \B$, we say that the {\em Empty Mapping Path} $[\;]$ of sort $X$ connects $y$ with itself in $\FD$.
    \item Two paths are equal, if they have the same length and are segmentwise equal.
    \item A mapping path is {\em proper} if there are no two distinct path segments that are weakly equal.
\end{itemize}
\end{defi}
Examples of mapping paths for graphs are depicted in Fig.\ref{fig:example1} (the complete meaning of the contents of Fig.\ref{fig:example1} will be explained in the next section): There are two paths (one along the dashed path segments, the other one along the dotted segments) both connecting vertex ''Sort'' with vertex ''Type''. Each arrow depicts a path segment with first component the arrow's source and third component its target. The middle component is annotated near the arrows, resp., their names will be explained in the next section, as well.
\par
For any $X\in\B$, any $i,j\in \I_0$ and any $z\in\FD_i(X)$, $z'\in\FD_j(X)$ we write
$\; z\equiv_X z' \mbox{ (}z\equiv_X^p z'\mbox{)}$,
if there is a mapping path (proper mapping path) of sort $X$ connecting $z$ with $z'$. It is easy to see that $\equiv \;=\; \equiv^p$ and that this relation is a congruence relation on $\coprod_{i\in\I_0}\FD_i$ (i.e., a family of equivalence relations $(\equiv_X)_{X\in\B}$ compatible with operations of $\B$), because paths can be concatenated and reversed. Moreover, it is well-known \cite{MacLane-1998} that the colimiting cocone of diagram $\FD:\I\to\G$ is given by
\begin{equation}\label{eqn:colimit-comp}
\FD\stackrel{\kappa}{\Rightarrow}(\coprod_{i\in\I_0}\FD_i)/\equiv \;\;=(\coprod_{i\in\I_0}\FD_i)/{\equiv^p}
\end{equation}
where $\kappa_i = [\;]_{\equiv}\circ \subseteq_i$ with $[\;]_{\equiv}$ the canonical morphism.
In the present paper we will show that mapping paths also play a crucial role for a simpler characterization of the Van Kampen property. The following examples hint at this connection.
\begin{exa}\label{ex:simple-examples} Let $\G = Set$.
\begin{enumerate}
  \item\label{enum:ex1_1} In \myref{eqn:double-id-coeq} there are proper mapping paths $[]$ and $[(\ast_2, d^{op}, \ast_1),$ $(\ast_1, d', \ast_2)]$ both connecting $\ast_2$ with itself.
  \item\label{enum:ex1_2} The shape graph $\xymatrix{1 & 0\ar[l]_{d}\ar[r]^{d'} & 2}$ yields pushouts. The easiest example of a non-VK pushout arises from $\FD_0 = \{x,y\}, \FD_1 =\{\ast_1\}, \FD_2 = \{\ast_2\}$, cf. \cite{Sobo2004}. In this case, we obtain two different proper mapping paths $[(\ast_1,d^{op},x),(x,d',\ast_2)] \mbox{ and } [(\ast_1,d^{op},y),(y,d',\ast_2)]$
  both connecting $\ast_1$ and $\ast_2$ in $\FD$.
  \item\label{enum:ex1_3} Let $\I = \xymatrix{\bullet\ar@(ul,dl)_d}$ consist of one vertex and one loop. I.e., diagrams depict endomorphisms $f:A\to A$. It is astonishing that even the colimiting cocone $\xymatrix{\FD\ar@{=>}[r] & \{\ast\}}$ with $\FD_d = id_{\{\ast\}}$ is not VK: Take $\FE = (\xymatrix{\{a,b\}\ar[r]^k & \{a,b\}})$ (with $k$ the non-identity bijection of $\{a,b\}$), $\tau: \{a,b\}\to \{\ast\}$, then $\FE$'s colimit is a singleton. In this example, we have two proper mapping paths $[]$ and $[(\ast, d, \ast)]$ in $\FD$ both connecting $\ast$ with itself. Note that this is just another presentation of example \myref{eqn:double-id-coeq}, since the colimit of $f:A\to A$ can be obtained by the coequalizer of$\Double{A}{f}{id_A}{A})$.
  \item\label{enum:ex1_4} $\FD = (\Double{\{x\}}{f}{g}{\{y,z\}})$ with $f(x) =y, g(x) = z$ has the VK property, which can be checked by elementary means based on Def.~\ref{def:VKCocone}. There is exactly one proper mapping path connecting $y$ and $z$, namely $[(y, f^{op}, x), (x, g, z)]$. Moreover, there is exactly one proper path connecting $y$ with itself (namely the empty one, the hypothetical path $[(y, f^{op}, x), (x, f, y)]$ is not proper, see Def.\ref{def:Weak-Mapping-Path}). In the same way $x$ has only one path back to itself, namely the empty one (the hypothetical path $[(x, f, y), (y, f^{op}, x)]$ is not proper).
\end{enumerate}
\end{exa}
As suggested by these examples, {\em uniqueness of proper mapping paths} between two elements of the same sort $X$ in the sets $(\FD_i(X))_{i\in\I_0}$ is a crucial feature for the Van Kampen property to hold. Indeed, we will prove
\begin{thm}[VK is Path Uniqueness]\label{theo:main}
Let $\G = \pre{\B}$ be a presheaf topos and $\FD:\I\to\G$ be a diagram. Let
$\;\FD\stackrel{\kappa}{\Rightarrow} \Delta S\;$
be a colimiting cocone. The cocone is a Van Kampen cocone if and only if for all $X\in \B$, all $i,j\in\I_0$ and all $z\in\FD_i(X)$, $z'\in\FD_j(X)$:
There are no two different proper mapping paths in $\FD$ connecting $z$ and $z'$.
\end{thm}
Since, in colimit computations, all mapping paths need to be computed (see \myref{eqn:colimit-comp}), and -- according to Theorem \ref{theo:main} -- the Van Kampen property can be checked by means of mapping paths, algorithmic verification of the Van Kampen property can be carried out in the background of colimit computation. A detailed elaboration of these combined computations is carried out in Sect.\ref{sec:practice}.
\subsection{Application of Theorem \ref{theo:main}}\label{sec:appl}
In order to demonstrate the benefits of the path uniqeness criterion, we consider a more substantial example than the ones in Example \ref{ex:simple-examples}. We let $\B = (\Double{E}{t}{s}{V})$ ($id_E$ and $id_V$ not shown), thus our base presheaf topos is $\G = Set^\B$, the category of directed multigraphs. In the sequel, we depict vertices as rectangles and edges are arrows pointing from its source to its target. In Fig.\ref{fig:example1} the three highlighted graphs\footnote{These are not just graphs since they contain "multiplicity constraints". They can be formalized, actually, as generalized sketches, i.e., graphs with diagrammatic predicates, in the sense of \cite{WD2008_a}. } $\FD_1$, $\FD_2$, and $\FD_3$ depict meta-models for type systems:
\begin{figure}[!htb]
\begin{center}
\includegraphics[height = 2.8in]{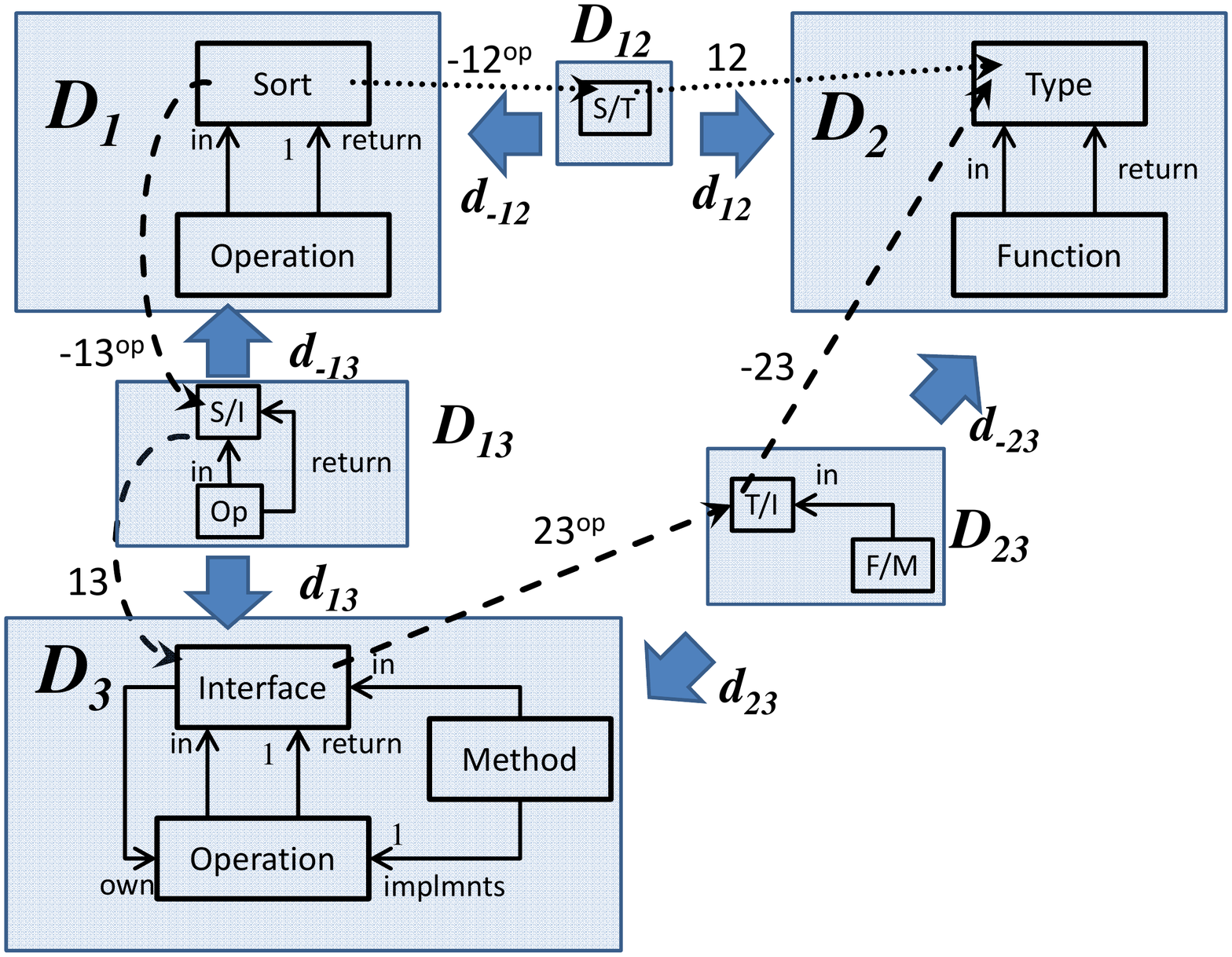}
\end{center}
\caption{A diagram of metamodels and two mapping paths}\label{fig:example1}
\end{figure}
\begin{itemize}
  \item $\FD_1$ represents parts of the domain of {\em algebraic specifications}: Operations have an arbitrary number of sort-typed input parameters and exactly one return parameter.
  \item In $\FD_2$ terminology of {\em abstract data types} is used: Functions have an arbitrary number of typed input and return parameters, resp.
  \item $\FD_3$ is the object-oriented view: Interfaces own operations, which have inputs and one return parameter typed in interfaces, resp. Methods implement operations, their input parameters may be of specialized type.
\end{itemize}
Figure \ref{fig:example1} represents a {\em multimodeling} scenario \cite{mehrdad-re07}. Reasoning about these collective models (the multimodel) as one artefact requires matching of different terminology of each of the model graphs: Sameness of terminology in graphs $\FD_1$ and $\FD_2$ is formally enabled by defining a relation on $\FD_1\times \FD_2$ by means of auxiliary graph $\FD_{12}$, which consists of exactly one vertex $S/T$, $d_{-12}(S/T) = Sort$, $d_{12}(S/T) = Type$, such that span $\FD_1\stackrel{d_{-12}}{\leftarrow}\FD_{12}\stackrel{d_{12}}{\to}\FD_2$ specifies sameness of terms ''Sort'' and ''Type'' in graphs $\FD_1$, $\FD_2$ and no other commonalities. In the same way span $\FD_1\stackrel{d_{-13}}{\leftarrow}\FD_{13}\stackrel{d_{13}}{\to}\FD_3$ specifies sameness of terms ''Sort'' and ''Interface'' (in $\FD_1$ and $\FD_3$) as well as ''Operation'' (in both graphs) together with the in- and return-relationships. Moreover, relation ''in'' of term ''Method'' in $\FD_3$ is declared to be equal to property ''in'' of term ''Function'' in $\FD_2$ via span $\FD_3\stackrel{d_{23}}{\leftarrow}\FD_{23}\stackrel{d_{-23}}{\to}\FD_2$.
\par
We now describe a scenario, in which colimit computation of the graphs in Fig.\ref{fig:example1} and amalgamation of instances typed over these graphs is important. It is common to reason about the multimodel by imposing constraints that spread over different models. We could, e.g., claim that ''The return type of a method's implemented operation (as specified in $\FD_3$) has to be contained in the list of return types of the corresponding function (as specified in $\FD_2$)''. In order to check this inter-model constraint, it is necessary to construct the diagram's colimit. Formally, for schema graph $\I = $
\[\xymatrix{
1 && 12\ar[ll]_{-12}\ar[rr]^{12} && 2 \\
& 13\ar[ul]^{-13}\ar[dr]_{13} && 23\ar[ur]_{-23}\ar[dl]^{23} & \\
&& 3 &&
}\]
we obtain diagram $\FD:\I\to \G$ and construct the colimiting cocone $\FD\stackrel{\kappa}{\Rightarrow}\Delta S$.
\par
Assume now that we want to check consistency of given typed instances $\tau_i:\FE_i\to \FD_i$, $i\in\{1,2,3\}$ against the above formulated constraint. For this we have to declare sameness of elements within $\FE_1, \FE_2$, and $\FE_3$ with the help of new relating typing morphisms $\tau_k:\FE_k\to \FD_k$, $k\in\{12,13,23\}$ and spans $\FE_1\stackrel{e_{-12}}{\leftarrow}\FE_{12}\stackrel{e_{12}}{\to}\FE_2$, $\FE_1\stackrel{e_{-13}}{\leftarrow}\FE_{13}\stackrel{e_{13}}{\to}\FE_3$, and $\FE_2\stackrel{e_{-23}}{\leftarrow}\FE_{23}\stackrel{e_{23}}{\to}\FE_3$. Of course, all $\tau_k$ have to be compatible with model matching and sameness declaration within $\FE_1, \FE_2$, and $\FE_3$, i.e., we obtain a natural transformation $\tau:\FE\Rightarrow\FD$ between diagrams of type $\I\to \G$. Consistency checking is then carried out by constructing the colimit object $K$ of $\FE$ and checking whether the resulting typing arrow $\sigma:K \to S$ fulfills the constraint, see \cite{mehrdad-re07}.
\par
Let us momentarily ignore constraint checking and just consider the relation between this amalgamated instance $\sigma$ and the original component instances $(\tau_i)_{i\in\{1,2,3,12,13,23\}}$: It is necessary to faithfully recover all $\tau_i$ from $\sigma$, otherwise we would loose information about the origin of the elements in the domain of $\sigma$. This means that we require, for all $i$, $\kappa_i^\ast(\sigma)\cong\tau_i$, i.e., the Van Kampen property for the cocone $\kappa$ has to hold. However, it turns out, that the property is violated: This can be seen by considering the following instance constellation (we write $\cb{x}{T}$ whenever $\tau_{\_}(x) = T$): Let $\FE_1(V) = \{\cb{s}{Sort}, \cb{s'}{Sort}\}, \FE_1(E) = \emptyset$, $\FE_2(V) = \{\cb{t_1}{Type}, \cb{t_2}{Type}\}, \FE_2(E) = \emptyset$, and $\FE_3(V) = \{\cb{\ol{i}}{Interface}, \cb{i}{Interface}\}, \FE_3(E) = \emptyset$. One may now declare sameness of elements within $\FE_1, \FE_2$, and $\FE_3$ as follows
\[s = t_1, s' = t_2 \mbox{ by span }(e_{-12},e_{12}); s = i, s' = \ol{i}\mbox{ by } (e_{-13},e_{13});  t_1 = \ol{i}, t_2 = i\mbox{ by }(e_{-23},e_{23}).\]
This is established as described above, e.g., graph $\FE_{12}$ consists of two vertices $\cb{1}{S/T}$ and $\cb{2}{S/T}$. Graph morphisms $e_{-12}$ maps $\cb{1}{S/T}\mapsto s$ and $\cb{2}{S/T}\mapsto s'$ wheras $e_{12}$ maps $\cb{1}{S/T}\mapsto t_1$ and $\cb{2}{S/T}\mapsto t_2$. We omit the obvious formal definitions of the other two spans.
\par
Unfortunately, by transitivity, this matching also yields $s=s'$, an unwanted anomaly. But in practice this effect may happen, if two modelers work separately: One modeler might define matches $(e_{-12},e_{12})$ and $(e_{-13},e_{13})$ and, independently and inadvertently, the second modeler defines the match $(e_{-23},e_{23})$. The inconsistent matching yields a colimit $K$ of $\FE$ with one vertex only, because each sort/type/interface is connected with each other along mapping paths. Clearly, $\kappa_i^\ast(\sigma)\not\cong\tau_i$ since $\kappa_i$ are monomorphisms, hence the domains of $\kappa_i^\ast(\sigma)$ are singleton sets, as well.
\par
In this simple example, a small instance constellation allowed for the detection of a VK violation witness. However, it is hard to determine such witnesses in more complex examples. In these cases, Theorem \ref{theo:main} is a more reliable indicator for VK validity or violation, because we do not need to find violating instance constellations. Instead, the violation of the VK property can be detected by analysing mapping path structures of the metamodels only. In the present example, the indicator are the {\em two different proper mapping paths} of sort $V$ shown in Fig.\ref{fig:example1} both connecting ''Sort'' and ''Type'' (one is depicted by dashed, the other one by dotted arrows) such that Theorem \ref{theo:main} immediately yields violation of the Van Kampen property.
At least from this example we derive the slogan that {\em the Van Kampen property holds, if there is no redundant matching information} in $\FD$. It is easy to see that the negative effect vanishes if we reduce the diagram accordingly, i.e. if we erase matching via $\FD_{12}$ since this information is already contained in the transitive closure of matchings $\FD_{13}$ and $\FD_{23}$. In this way, the above mentioned modelers can indeed work independently!


\section{An Outline of the Proof of Theorem \ref{theo:main}}\label{sec:to-colimits}
In this section, we sketch the main steps for the proof of our main theorem. Each step is given by a Lemma for which detailed proofs can be found in the technical report \cite{kw17}.
\subsection{Pushouts}\label{sec:chap2-VKPropDef}
Often, the Van Kampen property for pushouts is formulated as follows: A pushout of a diagram $\xymatrix{\FD_1 & \FD_0\ar[l]_{h_1}\ar[r]^{h_2} & \FD_2}$ is said to have the Van Kampen property if for any commutative cube
\[
\xymatrix{
& {\FE_0}\ar[dd]|{\hole}\ar[ld]\ar[rr]		& {}		& {\FE_2}\ar[dd]\ar[dl]	\\
{\FE_1}\ar[dd]\ar[rr]	&{} & {K}\ar[dd]^(0.3)\sigma& {} 	\\
& {\FD_0}\ar[rr]^(0.3){h_2}|{\hole}\ar[ld]_{h_1} 	& {} 		& {\FD_2} \ar[ld]^{\kappa_2}	\\
{\FD_1}\ar[rr]_{\kappa_1}	& 	& {S} 		& {} 	
}
\]
with this pushout in the bottom and back faces pullbacks, the front faces are pullbacks if and only if the top face is a pushout. In \cite{WK2013} we already stated a characterization of the Van Kampen property for pushouts based on this definition. It comes in terms of cyclic mapping structures within $\FD_0$:
\begin{defi}[Domain Cycle, \cite{L-10}]\label{def:domain-cycle}
Consider a span $\xymatrix{\FD_1 & \FD_0\ar[l]_{h_1}\ar[r]^{h_2} & \FD_2}$ in $\G = \pre{\B}$. For $X\in \B$ we call a sequence $[x_0,x_1,\ldots,x_{2k+1}]$ of elements of $\FD_0(X)$ a {\em domain cycle} (for the span $(h_1, h_2)$) (of sort $X$), if $k\in\N$ and the following conditions hold:
\begin{enumerate}
\item\label{cyc1} $\forall j\in \{0,1,\ldots, 2k+1\}: x_j\not=x_{j+1}$
\item\label{cyc2} $\forall i\in\{0,\ldots, k\}: h_1(x_{2i}) = h_1(x_{2i+1})$
\item\label{cyc3} $\forall i\in\{0,\ldots, k\}: h_2(x_{2i+1}) = h_2(x_{2i+2})$
\end{enumerate}
where ${2k+2} := 0$. A domain cycle is {\em proper} if $x_i\not= x_j$ for all $0\leq i<j\leq 2k+1$.
\end{defi}
The main outcome of \cite{WK2013} is the following fact:
\begin{lem}[Condition for VK Pushouts]\label{theo:VK-pushout}
A pushout
\[
\xymatrix{
\FD_0\ar[r]^{h_2}\ar[d]_{h_1} & \FD_2\ar[d]^{\kappa_2}
\\
\FD_1\ar[r]_{\kappa_1} & S
}
\]
in $\G = \pre{\B}$ is a Van Kampen cocone iff there is no proper domain cycle for $(h_1,h_2)$.\qed
\end{lem}
This result can be proven by means of elementary set-based arguments \cite{L-10}, but also by investigating forgetful functors between categories of descent data \cite{JT1994} for general topoi \cite{WK2013}.
\par
It is easy to see that the above definition for pushouts is an instance of the general definition of Van Kampen colimits in Def.~\ref{def:VKCocone}:
\begin{itemize}
  \item If the front faces are pullbacks, then the back faces are the result of applying $\kappa^\ast$. Then the counit $\e: \FF\circ \kappa^\ast\Rightarrow Id$ of adjunction is an isomorphism if and only if the cube's top face is already a pushout.
  \item If the top face is a pushout, then (up to isomorphism) $\sigma$ is the result of applying $\FF$. Then the unit $\eta: Id\Rightarrow \kappa^\ast\circ \FF$ is an isomorphism if and only if $\kappa^\ast(\sigma)$ produces the original cube up to isomorphism, i.e., the original front faces are pullbacks.
\end{itemize}
Hence, the two implications ''Front face pullbacks iff top face pushout'' actually reflect the two statements ''The counit is an isomorphism'' and ''The unit is an isomorphism''.
\par
Thus Lemma \ref{theo:VK-pushout} is a good starting point for the proof of Theorem \ref{theo:main}: We first transfer this knowledge to special mapping paths in coequalizer diagrams (Sect.\ref{sec:chap2-coequalizer}) and then from there to mapping paths in arbitrary colimits (Sect.\ref{sec:chap3-colimits}).

\subsection{From Pushouts to Coequalizers} \label{sec:chap2-coequalizer}
The transfer from pushouts to coequalizers is accomplished in two steps. The first step connects the VK property for coequalizers and pushouts:
\begin{lem}\label{lemma:coeq-pushout}
Let $\G$ be a general topos and $\Double{B}{f}{g}{D}$ be two arrows in $\G$. Let two arrows $\kappa_D:D\to S$ and $\kappa_B:B\to S$ be given such that the diagrams
\[\label{eqn:corresp-po}
\DoubleWithCocone{B}{f}{g}{D}{\kappa_D}{S}{\kappa_B}
\quad\quad
\xymatrix{
B + B \ar[r]^(0.6){[f,g]}\ar[d]_{[id,id]} & D\ar[d]^{\kappa_D} \\
B \ar[r]_{\kappa_B} & S
}
\]
are commutative, resp.
\begin{enumerate}
  \item\label{enum:coeq-pushout1} The left diagram is a coequalizer if and only if the right diagram is a pushout.
  \item\label{enum:coeq-pushout2} The left diagram is a VK cocone if and only if the right diagram is.
\end{enumerate}
\end{lem}
\begin{proof}
\ref{enum:coeq-pushout1} is well-known \cite{MacLane-1998}. \ref{enum:coeq-pushout2} is proven by means of Def.\ref{def:VKCocone}, where the transfer is possible, because topoi are (finitely) extensive, cf.\ Sect.\ref{sect:prelim}, and especially because of property \myref{eqn0.7}.
\end{proof}
The second step establishes a connection between domain cycles and mapping paths. It comes in terms of {\em disjoint} mapping paths, i.e., paths $P_1$ and $P_2$ for which none of the path segments in $P_1$ is weakly equal\footnote{Recall the definition of weak equality in Def.~\ref{def:Weak-Mapping-Path}.} to a path segment in $P_2$. The proof is rather technical and will be omitted, see \cite{kw17}, Lemma 14.
 \begin{lem}[Domain Cycles vs.\ Mapping Paths]\label{lemma:dom-cyc-and-mapp-paths}
Let $\G = \pre{\B}$, $f$ and $g$ as in Lemma \ref{lemma:coeq-pushout}, and $X\in \B$. There is a proper domain cycle of sort $X$ for $\xymatrix{B & B+B\ar[r]^(0.6){[f,g]}\ar[l]_(0.6){[id,id]} & D}$, if and only if there are $z,z'\in D(X)$ and two disjoint proper mapping paths connecting $z$ and $z'$. \qed
\end{lem}
Lemmas \ref{theo:VK-pushout}, \ref{lemma:coeq-pushout}, and \ref{lemma:dom-cyc-and-mapp-paths} yield
\begin{cor}[Condition for VK Coequalizers]\label{theo:VK-coeq}
Let $\G = \pre{\B}$, let $\DA$ be the schema graph $\Double{1}{d}{d'}{2}$, and $\ol{\FD}:\DA\to \G$. The coequalizer diagram
\[\DoubleWithCocone{\ol{\FD}_1}{\ol{\FD}_d}{\ol{\FD}_{d'}}{\ol{\FD}_2}{\kappa_2}{S}{\kappa_1}\]
has the Van Kampen property, if and only if for all $X\in \B$ and all $z,z'\in \ol{\FD}_2(X)$ : There are no two disjoint proper mapping paths of sort $X$ in $\ol{\FD}$ connecting $z$ and $z'$.\qed
\end{cor}
Recall the already made observations in Example \ref{ex:simple-examples}, \ref{enum:ex1_1}.\ and \ref{enum:ex1_4}., which confirm this statement.

\subsection{From Coequalizers to Colimits}\label{sec:chap3-colimits}
It is well-known \cite{MacLane-1998}, that the colimit of $\FD:\I\to \G$ can be computed by constructing the coequalizer of
\begin{equation}\label{eqn:colim-as-coeq1}
\Double
  {\coprod_{d\in\I_1}\FD_{s(d)}}
  {\vec{\FD_d}}
  {\vec{id}}
  {\coprod_{j\in\I_0}\FD_j,}
\end{equation}
where $\vec{\FD_d}$ and
  $\vec{id}$ are mediators out of the involved coproducts:
  \begin{equation}\label{eqn:mediators}\xymatrix{
  \coprod_{d\in\I_1}\FD_{s(d)} \ar[rr]^{\vec{\FD_d}} 			&&
  \coprod_{j\in\I_0}\FD_j &&
  \coprod_{d\in\I_1}\FD_{s(d)} \ar[rr]^{\vec{id}}		&&
  \coprod_{i\in\I_0}\FD_i
  \\
  \FD_i\ar[u]^{\subseteq_{i,d}}\ar[rr]^{\FD_d}							&& \FD_j\ar[u]_{\subseteq_j} &&
  \FD_i\ar[u]^{\subseteq_{i,d}}\ar@{=}[rr]								&& \FD_i\ar[u]_{\subseteq_i}
  }
  \end{equation}
  (for all edges $\Mapping{i}{d}{j}$ in $\I_1$).\footnote{Note that in the left coproduct of \myref{eqn:colim-as-coeq1} an object $\FD_i$ occurs as often as there are edges $d$ leaving $i$ in $\I$. Moreover, $\subseteq_{i,d}$ in \myref{eqn:mediators} denotes the embedding of $\FD_i$ into its appropriate copy, namely the source of $\FD_d$.}
\par
Let $\ol{\FD}: \DA\to \G$ be the functor mapping $\DA$ to the objects and arrows in  \myref{eqn:colim-as-coeq1}, where schema graph $\DA$ is given as before (cf.\ e.g. Cor.~\ref{theo:VK-coeq}). Then a technical analysis shows that we can combine mapping paths of $\FD$ with special mapping paths of $\ol{\FD}$ (again, we omit the proof and refer to \cite{kw17}):
\begin{lem}\label{lemma:path-osc}
Let $\G = Set^\B$ and $X\in \B$. The following statements are equivalent:
\begin{itemize}
  \item $\forall i,j\in\I_0 :\forall z\in\FD_i(X), \forall z'\in\FD_j(X)$: There are no two disjoint proper mapping paths in $\FD$ connecting $z$ and $z'$.
  \item $\forall z,z'\in\ol{\FD}_2(X) = \coprod_{j\in\I_0}\FD_j(X)$: There are no two disjoint proper mapping paths in $\ol{\FD}$ connecting $z$ and $z'$.\qed
  \end{itemize}
\end{lem}
The main part of the proof of Theorem \ref{theo:main} is to carry over the VK property for the coequalizer of  \myref{eqn:colim-as-coeq1} to its underlying colimiting diagram for $\FD$.
\begin{lem}\label{lemma:VK_colim_coeq}
For $\G:= Set^\B$, the cocone \myref{eqn:cocone} is VK if and only if the cocone
\begin{equation}\label{eqn:coeq}
\DoubleWithCocone{
\coprod_{d\in\I_1}\FD_{s(d)}}
  {\vec{\FD_d}}
  {\vec{id}}
  {\coprod_{j\in\I_0}\FD_j}
  {\ol{\kappa}}{S}{\ol{\kappa}'}
\end{equation}
resulting from constructing the coequalizer in \myref{eqn:colim-as-coeq1} is VK.
\end{lem}
{\em Proof}: Let $\kappa^\ast: \GO{S}\to \G^\I\Downarrow \FD$ be the functor introduced in \myref{equ:pulling-back-functor} and $\ol{\kappa}^\ast: \GO{S}\to \G^{\DA}\Downarrow \ol{\FD}$
be the corresponding functor for the colimiting cocone in \myref{eqn:coeq}. Using \myref{eqn0.5} and \myref{eqn0.7}, one can show that for each cartesian $\tau:\FE\Rightarrow\FD$ the squares
\[
\xymatrix{
\coprod_{d\in\I_1}\FE_{s(d)} \ar[rr]^{\vec{id}}
\ar[d]_{\coprod_{d\in\I_1}\tau_{s(d)}} && \coprod_{i\in\I_0}\FE_i\ar[d]^{\coprod_{i\in\I_0}\tau_i}\\
\coprod_{d\in\I_1}\FD_{s(d)} \ar[rr]^{\vec{id}} &&
\coprod_{i\in\I_0}\FD_i
}\quad\quad
\xymatrix{
\coprod_{d\in\I_1}\FE_{s(d)} \ar[rr]^{\vec{\FE_d}}
\ar[d]_{\coprod_{d\in\I_1}\tau_{s(d)}} && \coprod_{j\in\I_0}\FE_j\ar[d]^{\coprod_{j\in\I_0}\tau_j}\\
\coprod_{d\in\I_1}\FD_{s(d)} \ar[rr]^{\vec{\FD_d}} &&
\coprod_{j\in\I_0}\FD_j
}
\]
are pullbacks, i.e., there is the assignment $\tau\mapsto (\coprod_{d\in\I_1}\tau_{s(d)}, \coprod_{i\in\I_0}\tau_{i})$. It can be shown with elementary arguments that it extends to an equivalence of categories:
\[\phi:\G^{\I}\Downarrow{\FD}\cong\G^{\DA}\Downarrow\ol{\FD}.\]
Moreover, the colimit construction principle, see \myref{eqn:colim-as-coeq1}, yields commutativity of
\[
\xymatrix{
&\GO{S}\ar[dl]_{\kappa^\ast} \ar[dr]^{\ol{\kappa}^\ast} &\\
\G^\I\Downarrow \FD \ar_{\phi}[rr]&& \G^{\DA}\Downarrow \ol{\FD}
}
\]
up to natural isomorphism, hence, by Def.~\ref{def:VKCocone}, the desired result. \qed
\subsection{Combining the Results}
We are now ready to prove Theorem \ref{theo:main} for disjoint proper mapping paths. This follows by combining Lemma \ref{lemma:VK_colim_coeq}, Corollary \ref{theo:VK-coeq} and Lemma \ref{lemma:path-osc}. Afterwards we can get rid of disjointness by showing that any two proper mapping paths connecting the same two elements also admit two disjoint proper paths (probably connecting two different elements). \qed


\section{Counterexamples in other Categories}\label{sec:counterex}
The valuable implication in the equivalence statement of Theorem \ref{theo:main} is ''Path-Uniqueness implies the Van Kampen Property''. In this section, we show that this implication immediately breaks, if we leave the universe of presheaf topoi, i.e.\ there are simple counter-examples in categories that are simililar to, but don't meet all axioms of presheaf topoi. We separate the examples depending on whether unit $\eta$ or counit $\varepsilon$ of adjunction $\kappa_\ast\dashv \kappa^\ast$ fails to be isomorphic (cf.\ Def.\ref{def:VKCocone} and ensuing remarks) although path uniqueness holds. 
\par
For the sake of completeness we finally give an example for a violation of the reverse direction of the equivalence in Theorem \ref{theo:main}. It illustrates a situation of a Van Kampen cocone, although path uniqueness does not hold. 

\subsection{Path Uniqueness, but no Isomorphic Unit}
Consider the category $HSet$, whose objects are sets, where in each set at most one element may be highlighted (we also say ''marked''). We denote such a set with $(x,A)$ (indicating that $x\in A$ is marked). If no element is highlighted, we write $(\bot,A)$. Formally these sets are {\em partial} algebras \cite{EGW1998,Rei87,Wol90,WKWC94} w.r.t.\ the algebraic signature $\Sigma$ with one sort and one constant symbol $h$ (''$h$ighlighted''), such that in any $\Sigma$-algebra with carrier set $A$ there is a constant, which can be understood as a partial map $h:\{\ast\}\to A$, i.e.\ there may or may not be a marked element depending on whether $h$ is defined for $\ast$ or not. A morphism between $(x,A)$ and $(y,B)$ is a total mapping $f:A\to B$ for which the additional property 
\[x\not=\bot\Rightarrow (y\not=\bot \mbox{ and }f(x) = y)\] 
holds. It can be shown \cite{EGW1998} that $HSet$ possesses all limits and colimits.   
\par 
In Fig.\ref{fig:example3}, the bottom face is a pushout (note that the one element set $\{\ast\}$ arises from the fact that there can not be more than one marked element). The behaviour of each mapping should be obvious. Clearly, the back faces are pullbacks. Although the top face is a pushout, the front faces fail to be pullbacks. Path uniqueness for $(h_1, h_2)$ is satisfied trivially.   
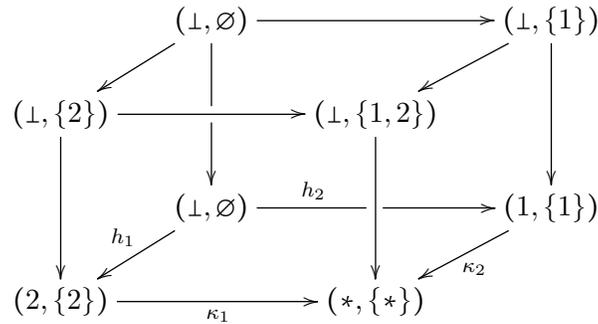
\begin{figure}[!htb]
\[
\xymatrix{
& {(\bot,\emptyset)}\ar[dd]|{\hole}\ar[ld]\ar[rr]		& {}		& (\bot,\{1\})\ar[dd]\ar[dl]	\\
(\bot,\{2\})\ar[dd]\ar[rr]	&{} & (\bot,\{1,2\}) \ar[dd]& {} 	\\
& (\bot,\emptyset)\ar[rr]^(0.3){h_2}|(0.48){\hole}\ar[ld]_{h_1} 	& {} 		& (1,\{1\}) \ar[ld]^{\kappa_2}	\\
(2,\{2\})\ar[rr]_{\kappa_1}	& 	
& (\ast,\{\ast\}) 		& {} 	
}   
\]
\caption{Path Uniqueness, but not VK, I}\label{fig:example3}
\end{figure}

\subsection{Path Uniqueness but no Isomorphic Co-Unit}
By the remarks in Sect.\ref{sec:main}, this can only be the case, if the pullback functor possesses no right-adjoint. For this, we consider the category $\C = CAT$ of all small categories together with functors between them. The following example in $\C$ simultaneously shows (1) a non-VK-pushout with path uniqueness and (2) that the pullback functor has no right-adjoint.
\begin{figure}[!htb]
\[ 
\xymatrix{ 
& {\ingr{0.35}{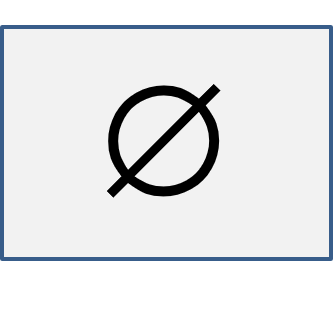}}\ar[dd]|(0.42){\hole}\ar[ld]\ar@<2ex>[rr]		& {}		& {\ingr{0.35}{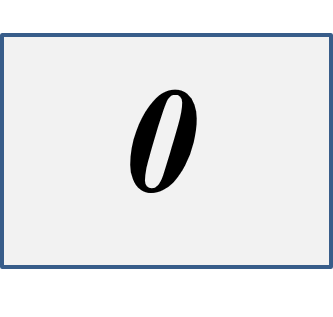}}\ar[dd]\ar[dl]	\\
{\ingr{0.35}{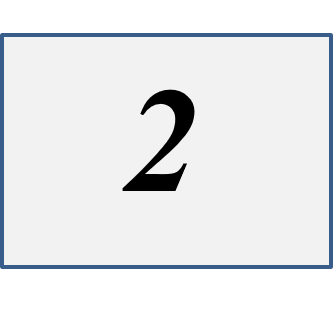}}\ar[dd]\ar@<2ex>[rr]	&{} & {\ingr{0.35}{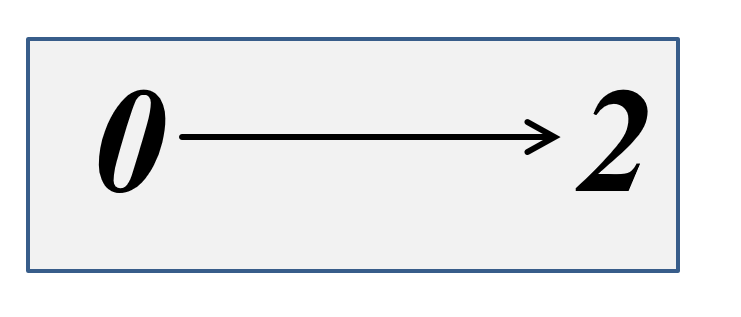}} \ar[dd]^(0.3)\sigma& {} 	\\
& {\ingr{0.35}{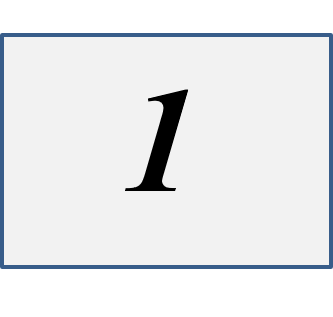}}\ar@<2ex>[rr]^(0.3){h_2}|(0.46){\hole}\ar[ld]_{h_1} 	& {} 		& {\ingr{0.35}{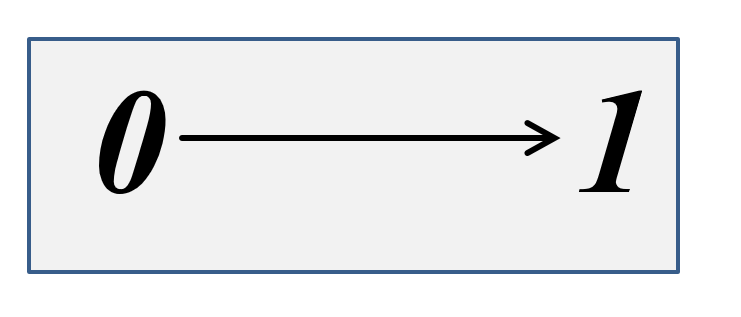}} \ar[ld]^{\kappa_2}	\\
\ingr{0.35}{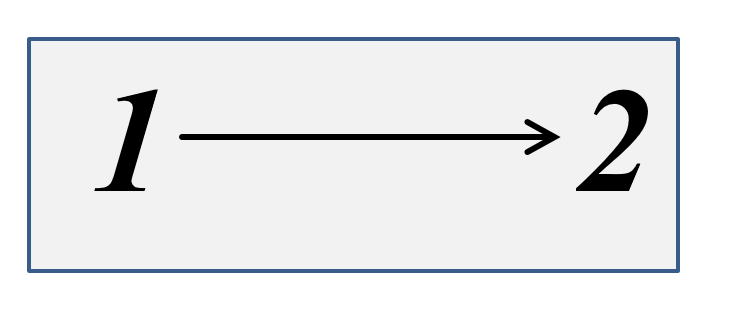}\ar@<2ex>[rr]_{\kappa_1}	& 	
& {\ingr{0.35}{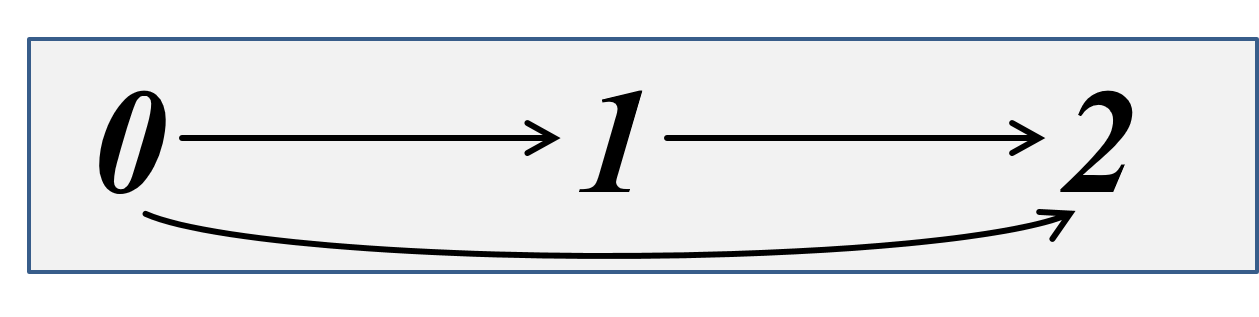}} 		& {} 	
}   
\]
\caption{Path Uniqueness, but not VK, II}\label{fig:example2}
\end{figure}
In Fig.\ref{fig:example2} each shaded rectangle shows a category, e.g.\ in the bottom there is a category with exactly one element $1$, two categories with one non-identity arrow between objects $0$ and $1$, $1$ and $2$, resp., and a category with three elements $0$,$1$, and $2$ where the arrow from $0$ to $2$ is the composition of the other two. In all cases, identity arrows are not shown. Arrows between the shaded rectangles are functors that map according to the numbering. This example has already been discussed in \cite{Sobo2004}.
\par
Obviously, the bottom square is a pushout in $\C$ and the back faces are pullbacks ($\emptyset$ denotes the empty category). Moreover, the front faces are pullbacks, but the top face fails to be a pushout (the pushout should be the category consisting of two objects $0$ and $2$ and no non-identity arrow), hence $\varepsilon$ is not an isomorphism. The pullback functor $\sigma^\ast$ maps the bottom square to the top square. If it would possess a right-adjoint, it would preserve the bottom colimit, which is not the case.
\par
Clearly, path uniqueness holds in span $(h_1,h_2)$, because both functors are injective on objects and arrows, resp. Note that we treat $\C$ as a many-sorted algebra with carrier sets $O$(bjects) and $(Hom(x,y))_{x,y\in O}$ (Hom-Sets) and (in contrast to graphs) with a family of binary operations $(\circ_{x,y,z}: Hom(y,z)\times Hom(x,y)\to Hom(x,z))_{x,y,z\in O}$, a family of constants $(id_x\in Hom(x,x))_{x\in O}$, and the appropriate monoidal axioms for neutrality (of identities) and associativity (of operations $\circ_{\_,\_,\_}$).

\subsection{Van Kampen holds, but Path Uniqueness is Violated}
This requires an artificial and radical narrowing of the size of the underlying category. We consider a category $\C$ which has sets $X$ with two constants $0$ and $1$ (actually an algebra for a signature with one sort and two constant symbols) subject to the following axiom
\[\forall x\in X: x = 0\vee x = 1.\]
Let sets $X$ and $Y$ with constants $0^X, 1^X$ and $0^Y, 1^Y$ be given, then an arrow from $X$ to $Y$ is a mapping $f:X\to Y$, which preserves constants, i.e.\ $f(0^X) = 0^Y$ and $f(1^X) = 1^Y$. It is not forbidden that the two constants coincide, such that the above equation reduces this category to two sets $2:=\{0,1\}$ and $1:=\{01\}$ (in set $1$, the constants coincide). It can further be observed that the only non-identity arrow is $2\to 1$. Furthermore, the only pullback of a co-span with two non-identity arrows is 
\[\xymatrix{
2\ar@{=}[r]\ar@{=}[d] & 2\ar[d] \\
2\ar[r] & 1 \\
}
\]
It can even be shown that this (very small) category has all limits and colimits (e.g.\ the initial object is $2$, terminal object is $1$).
\par 
We can now pick up the coequalizer example \myref{eqn:double-id-coeq} of Sect.\ref{sec:VKColimits}. There is the coequalizer diagram $\Double{1}{}{}{1}\Mapping{}{}{(1=:S)}$ where all arrows are identities. As pointed out in Example \ref{ex:simple-examples}, \ref{enum:ex1_1} path uniqueness is violated. 
\par
In each pullback pair 
\[\xymatrix{
A\ar[r]\ar@<1.ex>[r]\ar[d] & B\ar[d] \\
1\ar[r]\ar@<1.ex>[r] & 1 \\
}
\]
we must have $A = B$ and the two top mappings must both be identities. Note that the non-identical bijection of $\{0,1\}$ as in \myref{eqn:double-id-coeq} can no longer occur. Thus, only two such pullback pairs exist. There are also exactly two objects in $\CO{S}$, namely $id_1$ and $2\to 1$, and it is easy to see that they correspond to each other via $\kappa^\ast$ and $\kappa_\ast$. Hence the coequalizer has the Van Kampen property, although path uniqueness is violated.


\section{Practical Guidelines}\label{sec:practice}
A main use case of the previous results can be found in software engineering and especially in model-driven software designs: Components of diagrams are models (e.g. data models or metamodels which govern the admissible structure of models), morphisms are relations between the models. As described in the introduction, model assembly is often important (cf. Sect.\ref{sec:appl}). It shall not only be carried out on a syntactical level (models), but in the same way on the semantical level (instances) such that assembled instances can correctly be decomposed into their original instances by pullback. 
\par 
It is a goal to efficiently verify whether {\em compositionality} holds, i.e.\ whether the Van Kampen property is satisfied. Since model composition always requires computation of colimits, VK-verification at the same time is desirable. In this section we will describe (1) an efficient colimit computation and (2) how to simultaneously check the VK-property. 
\par
To further reduce verification effort, we will first look for criteria to decide, for a given diagram, if VK holds or not, without checking explicitly the path conditions of Theorem \ref{theo:main}. Moreover, we will show that -- in cases where the path condition is needed -- one does not need to check the condition for {\em all} $i,j\in\I_0$ but only for a smaller subset of indices. All investigations lead to a decision diagram, which guides a modeler who has to decide whether a given diagram has the VK property or not. This diagram is given in the end of this section, see Fig,~\ref{fig:decision-diagram}.
\par 
All additional constructions are  elementary and will not be elaborated in detail. Instead we refer the reader to the corresponding technical report \cite{kw17}.
\subsection{Relevant Types of Mapping Paths}
We will discuss now what kinds of paths and what pairs of paths we really need to check in practice. The attentive reader may have noticed already that properness excludes cycles w.r.t.\ path segments but does not exclude cycles w.r.t.\ elements.
Let $P=[(y_0, \delta_0, y_1), (y_1, \delta_1, y_2), \ldots, (y_{n-1},\delta_{n-1}, y_n)]$ be a mapping path in a diagram $\FD:\I\to\pre{\B}$ with finite $\I$. By $\iota_i$, we denote the unique vertex in $\I_0$ with $y_i\in\FD_{\iota_i}$.

\begin{defi}\label{def:inner-cycle-free}
A mapping path is called {\em inner-cycle free}, if for all indices $0\leq i < j \leq n$ with $j-i\leq n-1$:  $y_i\not=y_j$. 
\end{defi}
Empty mapping paths or paths of length 1 are inner-cycle free by definition. Note, that we allow $P$ to be an "outer" cycle, i.e., $y_0=y_n$. An elementary construction \cite{kw17} shows that each path can be reduced to an inner-cycle free non-empty (sub-)path. The reduction preserves properness and disjointness of mapping paths. Moreover, one can easily see that two different proper paths from $z$ to $z'$ even yield two {\em disjoint} proper paths (probably between other elements). Thus we obtain the following corollary of Theorem \ref{theo:main}: 
\begin{cor}\label{cor:main}
Let $\G = \pre{\B}$ be a presheaf topos and $\FD:\I\to\G$ be a diagram with $\I$ a finite directed multigraph. Let
\[\FD\stackrel{\kappa}{\Rightarrow} \Delta S\]
be a colimiting cocone. The following are equivalent: 
\[\begin{array}{cll}
(1) & \mbox{The cocone is a Van Kampen cocone}\\
(2) & \forall X\in \B, i,j\in\I_0, z\in\FD_i(X), z'\in\FD_j(X):&
\mbox{There are no two different proper paths} \\
 & & \mbox{from $z$ to $z'$} \hfill\\
(3) & \forall X\in \B, i,j\in\I_0, z\in\FD_i(X), z'\in\FD_j(X):&
\mbox{There are no two disjoint proper paths} \\
 & & \mbox{from $z$ to $z'$} \hfill\\
(4) & \forall X\in \B, i,j\in\I_0, z\in\FD_i(X), z'\in\FD_j(X):&
\mbox{There are no two disjoint inner-cycle free} \\
 & & \mbox{proper paths from $z$ to $z'$} \hfill   
\end{array}
\]
\end{cor}
In addition to a better variety of VK characterisations, this result also provides a high degree of freedom for the implementation of algorithms: The result does not depend on whether an algorithm generates only disjoint pairs of paths, or also builds paths with inner cycles. However, every algorithm based on Corollary \ref{cor:main} still has to traverse all components $(\FD_i)_{i\in\I_0}$. The next sections will explain why the traversal of components can significantly be reduced.  

\subsection{Cyclic Shape Graphs}\label{sec:cyclic}
A {\em directed cycle} in $\I$ is a set of pairwise distinct edges $d_0, \ldots, d_{n-1}$ in $\I_1$ for some $n\ge 1$ with $t(d_{i-1}) = s(d_{i\,mod\,n})$ ($i\in\{1, \ldots, n\}$). If $n=1$, the cycle is also called a loop (cf. Example \ref{ex:simple-examples}, \ref{enum:ex1_3}). Let $ends(d) = \{s(d), t(d)\}$ be the set of endpoints of an edge $d$, then an {\em undirected cycle} in $\I$ is a set of pairwise distinct edges $d_0, \ldots, d_{n-1}$ in $\I_1$ for some $n\ge 2$ with $ends(d_{i-1})\cap ends(d_{i\,mod\,n}) \not =\emptyset$ ($i\in\{1, \ldots, n\}$). An example for an undirected cycle is the shape graph $\DA$ and also the shape graph in the example of Sect.\ref{sec:appl}.  
\par 
An important observation is that VK is violated, if $\I$ posseses a directed cycle  
\[p=(i_0\stackrel{d_0}{\to}i_1\stackrel{d_1}{\to}\ldots\stackrel{d_{n-1}}{\to}i_n = i_0)\]
and if for some $X\in \B$ and some $i\in\{i_0,\ldots, i_n\}$ the component $\FD_i(X)$ is a finite set. To see this, let w.l.o.g.\ $i= i_0$ and 
\[\FD_p:=\FD_{d_{n-1}}\circ\ldots\circ\FD_{d_1}\circ\FD_{d_0}:\FD_{i_0}(X)\to \FD_{i_n}(X)  = \FD_{i_0}(X).\]
be the corresponding composed function. Obviously, there exist $y\in\FD_{i_0}(X)$ and $1\leq k\leq |\FD_{i_0}(X)|$ such that $\FD_p^k(y)=y$, where the mappings can be chosen such that this results in a non-empty proper mapping path connecting $y$ with itself: Because the empty path also connects $y$ with itself (cf. Def.\ref{def:Weak-Mapping-Path}), VK is violated by Cor.\ref{cor:main}.


\subsection{Specialized Construction of Colimits}\label{sec:colim-simple}
In this section, we prepare a general and more efficient algorithm for VK verification, which runs in the background of a colimit computation. For this we need to distinguish between different characteristics of shape graph $\I$ and derive from that a specialized colimit construction. 
\par
The universal construction recipe for colimits \myref{eqn:colim-as-coeq1} shows how to construct the colimit of any diagram in a uniform way by means of a coequalizer. It allows to prove general results about colimits (as we have demonstrated in the previous sections). Especially, in case of graphs $\I$ with infinite descending chains and/or with directed cycles, this universal recipe is the best we have. E.g. colimit construction of 
\[
\xymatrix{
\FD_1\ar@/^/[r]^{\FD_d}
&  \FD_2\ar@/^/[l]^{\FD_{d'}},
}
\]
and its VK verification is based on mapping paths within the coproduct $\FD_1 + \FD_2$ and there is no way of minimizing the space for investigation.
\par 
A first step in simplifying colimit computation and VK verification can be found, if we consider coequalizers. They can be investigated within a smaller space: A coequalizer is not VK, if we can find for the corresponding diagram $\FD:\DA\to\pre{\B}$ a sort $X\in\B$ and an element $y\in\FD_1(X)$ such that $\FD_d(y)=\FD_{d'}(y)$, thus reducing investigations to $\FD_1$. However, even if $\FD_d$ and $\FD_{d'}$ do have sortwise disjoint images, VK may be violated. Note, that those image disjoint diagrams are exactly the diagrams we obtain when constructing pushouts, in the traditional way, by means of sums and coequalizer. To have VK we can require, in addition, that $\FD_d$ and $\FD_{d'}$ are monic, since then $[\FD_d, \FD_{d'}]$ is monic and hence the pushout along this mono is VK (because each topos is adhesive \cite{LS06, Sobo2004}). In these special cases, this provides again significant simplification. 
\par 
Even in the cases left unclear, the check of the VK property for coequalizers must not utilize the condition in Cor.\ref{cor:main}, which is based on the universal construction recipe \myref{eqn:colim-as-coeq1} for colimits\footnote{Note, that we could apply the universal construction recipe again to the diagram in \myref{eqn:colim-as-coeq1} and so on.}. Instead, we can use directly the condition in Corollary \ref{theo:VK-coeq}, i.e.\ we need to check the condition in Cor.\ref{cor:main} only for the case $i=j=2$, thus reducing investigations to $\FD_2$. We will see in the forthcoming parts that all these effects can often be used in more general cases to reduce analysis effort. 
\par 
Beside these effects, there may be components that do not contribute to the construction of the colimit at all. In many cases, this simplifies colimit construction, because it is not necessary to compute a quotient of the {\em entire} coproduct $\coprod_{j\in\I_0}\FD_j$. Moreover, we will investigate how certain further assumptions on the properties of arrows in $\FD$ (image-disjointness, injectivity) also simplify the algorithm. We will discuss some further examples to motivate the announced specialized and practical construction of colimits.
\par
Since the case of directed cycles can immediately be handled in practical situations\footnote{
\ldots where, presumably, components are finite artefacts \ldots}, we assume from now on that {\em $\I$ is finite and has no directed cycles}.

\paragraph{{\em Irrelevant components}:} For all indices in $\I$ with no incoming and exactly one outgoing edge, the corresponding component does not contribute to the construction of the colimit. Typical examples are 
\[
\xymatrix{
\FD_1\ar[r]^{\FD_d} & \FD_2
&& \FD_1\ar[r]^{\FD_{d_1}} & \FD_3 &  \FD_2\ar[l]_{\FD_{d_2}},
}
\]
(in an arbitrary category). For the left diagram $\FD_2$ can be taken as the colimit object and we can set $\kappa_2:=id_{\FD_2}$, $\kappa_1:=\FD_d$. For the right diagram $\FD_3$ can serve as colimit object and we have $\kappa_3:=id_{\FD_3}$, $\kappa_1:=\FD_{d_1}$, $\kappa_2:=\FD_{d_2}$.

\paragraph{{\em Jump-over components}:} For all indices in $\I$ with exactly one incoming and one outgoing edge we can jump over the corresponding component. As examples, we consider the diagrams
\[
\xymatrix{
\FD_0\ar[r]^{\FD_{d_1}} & \FD_1\ar[r]^{\FD_{d_2}} & \FD_2
&& \FD_3 & \FD_0\ar[r]^{\FD_{d_1}}\ar[l]_{\FD_{d_3}} & \FD_1\ar[r]^{\FD_{d_2}} &  \FD_2,
}
\]
(in an arbitrary category). For the left diagram, $\FD_2$ can be taken as the colimit object and we have $\kappa_2:=id_{\FD_2}$, $\kappa_1:=\FD_{d_2}$, $\kappa_0:=\FD_{d_2}\circ\FD_{d_1}$. For the right diagram the colimit object is obtained by the pushout of $\FD_{d_2}\circ\FD_{d_1}:\FD_0\to\FD_2$ and $\FD_{d_3}:\FD_0\to\FD_3$. The missing injections are given by $\kappa_1:=\kappa_2\circ\FD_{d_2}$ and $\kappa_0:=\kappa_2\circ\FD_{d_2}\circ\FD_{d_1}(=\kappa_3\circ\FD_{d_3})$.

\paragraph{{\em Minimal components}:} We consider diagrams for coequalizer and pushouts, respectively,
\[
\xymatrix{
\FD_1\ar@/^/[r]^{\FD_d}\ar@/_/[r]_{\FD_{d'}} &  \FD_2
&& \FD_1 & \FD_0\ar[l]_{\FD_d}\ar[r]^{\FD_{d'}} & \FD_2
}
\]
and the corresponding diagrams according to the universal recipe \myref{eqn:colim-as-coeq1}
\[
\xymatrix{
   \FD_1+\FD_1\ar@<-1ex>[rr]_{[\subseteq_2\circ\FD_d,\subseteq_2\circ\FD_{d'}]}
           \ar@<1ex>[rr]^{[\subseteq_1,\subseteq_1]}
&& \FD_1+\FD_2
&& \FD_0+\FD_0\ar@<-1ex>[rr]_{[\subseteq_1\circ\FD_d,\subseteq_2\circ\FD_{d'}]}
           \ar@<1ex>[rr]^{[\subseteq_0,\subseteq_0]}
&& \FD_0+\FD_1+\FD_2
}
\]
In case of coequalizer we construct, usually, a quotient of $\FD_2$ and not of $\FD_1+\FD_2$ and, in case of pushouts we construct a quotient of $\FD_1+\FD_2$ and not of $\FD_0+\FD_1+\FD_2$. Also in the example in Figure \ref{fig:example1} we factorize the sum $\FD_1+\FD_2+\FD_3$ and not the sum of all 6 components. In all three cases we build first the coproduct of all minimal components (see Def.~\ref{def:minimal-index}) and construct then a quotient of this restricted coproduct.
\begin{defi}[Minimal Components]\label{def:minimal-index}
For a finite directed multigraph $\I$ we denote by $Min(\I)$ the set of all {\em(local) minimal indices}, i.e., of all vertices in $\I_0$ without outgoing edges. For a diagram $\FD:\I\to\pre{\B}$ we say that $\FD_i$ is a {\em minimal component} if $i\in Min(\I)$.
\end{defi}
Since $\I$ is finite and has no directed cycles, each index is either minimal or there exists a non-empty finite sequence of edges to at least one minimal index. This fact will be  used several times in the sequel. To construct the colimit of a diagram with finite $\I$ with no directed cycles we need, in practice, only the minimal components as outlined below.

\paragraph{{\em Branching components}:} The essence in constructing the colimit of a diagram is to converge diverging branches in the diagram (in a minimal way). In presheaf topoi this can be done by sortwise identifying certain elements. What elements, however, have to be identified? 
\par
In case of coequalizers we have to identify for all sorts $X\in\B$ and all $y\in\FD_1(X)$ the two elements $\FD_d(y)$ and $\FD_{d'}(y)$ in $\FD_2(X)$. These primary identifications induce further identifications when we construct the smallest congruence in $\FD_2$ comprising all these primary identifications. For pushouts we have to identify for all $X\in\B$ and all $y\in\FD_0(X)$ the two elements $\FD_d(y)$ and $\FD_{d'}(y)$ seen as elements in $\FD_1(X)+\FD_2(X)$. In this case, we construct then the smallest congruence in $\FD_1+\FD_2$ comprising all the primary identifications.
\par
Now we look at slightly more general diagrams. First, we consider three parallel arrows.
\xymatrixrowsep{3pc}
\xymatrixcolsep{3pc}
\[
\xymatrix{
\FD_1\ar@/^1pc/[r]^{\FD_{d_1}}\ar[r]|{\FD_{d_2}}\ar@/_1pc/[r]_{\FD_{d_3}} &  \FD_2
}
\]
\xymatrixrowsep{1.5pc} 
\xymatrixcolsep{1.5pc}
In this case, we have to identify for all sorts $X\in\B$ and all $y\in\FD_1(X)$ the three elements $\FD_{d_1}(y)$, $\FD_{d_2}(y)$ and $\FD_{d_3}(y)$ in $\FD_2(X)$, and then we construct the smallest congruence in $\FD_2$ comprising these identifications. Second, we consider two generalizations of pushouts
\[
\xymatrix{
&  \FD_I\ar[ld]_{\FD_{is}}\ar[rd]^{\FD_{ib}} &
&& \FD_1 & \FD_{13}\ar[l]_{\FD_{d_{-13}}}\ar[r]^{\FD_{d_{13}}} &  \FD_3
\\
   \FD_S & \FD_P\ar[l]_{\FD_{m}}\ar[r]^{\FD_{r}} & \FD_B
&& \FD_{12}\ar[u]^{\FD_{d_{-12}}}\ar[r]^{\FD_{d_{12}}} & \FD_{2} &  \FD_{23}\ar[l]_{\FD_{d_{-23}}}\ar[u]_{\FD_{d_{23}}}
}
\]
A situation, as in the left diagram, appears, for example, if we want to avoid that the instantiation of a "parameterized specification" $\Mapping{\FD_P}{\FD_r}{\FD_B}$ via a "match" $\Mapping{\FD_P}{\FD_m}{\FD_S}$ generates two copies of a specification $\FD_I$ that had been imported as well by the "body" $\FD_B$ of the parameterized specification as by the "actual parameter" $\FD_S$.
The diagram on the right is taken from the example in Figure \ref{fig:example1}.

In the left diagram we have to identify for all $X\in\B$ and all $y\in\FD_P(X)$ the two elements $\FD_m(y)$ and $\FD_{r}(y)$ seen as elements in $\FD_S(X)+\FD_B(X)$. In addition, we have to identify for all $z\in\FD_I(X)$ the two elements $\FD_{is}(z)$ and $\FD_{ib}(z)$, again seen as elements in $\FD_S(X)+\FD_B(X)$.

In the right diagram, we have, analogously, that the elements in $\FD_{12}$ force identifications of elements in $\FD_1$ and $\FD_2$, seen as elements in $\FD_1+\FD_2+\FD_3$, the elements $\FD_{13}$ force identifications of elements in $\FD_1$ and $\FD_3$, seen as elements in $\FD_1+\FD_2+\FD_3$, and the elements in $\FD_{23}$ force identifications of elements in $\FD_2$ and $\FD_3$, seen as elements in $\FD_1+\FD_2+\FD_3$.

Generalizing the examples we want to coin the following definition.
\begin{defi}[Branching Components]\label{def:branching-index}
For a finite directed multigraph $\I$ we denote by $Br(\I)$ the set of all {\em branching indices}, i.e., of all indices with, at least, two outgoing edges. For a diagram $\FD:\I\to\pre{\B}$ we say that $\FD_i$ is a {\em branching component} if $i\in Br(\I)$.
\par
A sequence of edges $p=(i_0\stackrel{d_0}{\to}i_1\stackrel{d_1}{\to}\ldots\stackrel{d_{n-1}}{\to}i_n)$\; in $\I$  is called a {\em branch}, if $i_0\in Br(\I)$ and $i_n\in Min(\I)$.
\end{defi}
Note, that $Br(\I)$ and $Min(\I)$ are disjoint by definition. Thus any branch has at least length 1. Note further, that branching indices can be connected in $\I$, in contrast to minimal indices. To illustrate our discussion and definitions we consider a simple toy example.
\begin{exa}\label{ex:shape-graph}
Let $\I =$
\[
\xymatrix{
  1 \ar[r]^a & 2\ar[ld]_b & 4\ar[r]_c\ar@/^1pc/[rr]^e & 5\ar[r]_d & 6\ar[ld]^g\ar[d]^h
\\
  3 && 7\ar[u]^f & 8 & 9\ar[r]^l & 10
}
\]
Here we have $Min(\I)=\{3,8,10\}$, $Br(\I)=\{4,6\}$ and the four branches $(4\stackrel{c}{\to}5\stackrel{d}{\to}6\stackrel{g}{\to}8)$, $(4\stackrel{c}{\to}5\stackrel{d}{\to}6\stackrel{h}{\to}9\stackrel{l}{\to}10)$, $(4\stackrel{e}{\to}6\stackrel{g}{\to}8)$, and $(4\stackrel{e}{\to}6\stackrel{h}{\to}9\stackrel{l}{\to}10)$.

The indices $1$ and $7$ are irrelevant and we can jump over the indices $5$ and $9$. We can also jump over the index $2$, but since $1$ is irrelevant also $2$ becomes irrelevant. Be aware that the edges $c,d,e$ constitute an undirected  cycle in $\I$.
\end{exa}
Branching indices give rise to special positions in mapping paths: 
\begin{defi}[Branching Position in Proper Paths]\label{def:turning}
For a pair of subsequent segments 
\[(y_{j-1},d^{op},y_j),(y_j,d',y_{j+1})\]
of a proper mapping path $P$ (hence $d\not=d'$) the position $j$ is called a {\em branching position} of $P$. Consequently $\iota_j$ is a branching index of $\I$. 
\end{defi} 

\paragraph{{\em A specialized construction of colimits}:}
Now, we have everything at hand to describe a construction of colimits in presheaf topoi. Let $\I$ be a finite directed multigraph with no directed cycles. Then we can construct the colimit of a diagram $\FD:\I\to\pre{\B}$ as follows:
\begin{enumerate}
  \item Construct the coproduct $\coprod_{i\in Min(\I)}\FD_i$ (by sortwise coproducts in $Set$).
  \item \label{enum:spec-colim2} For each pair $p=(i\stackrel{d}{\to}\ldots\to j)$, $p'=(i\stackrel{d'}{\to}\ldots\to j')$ of branches in $\I$ with common source $i\in Br(\I)$ and $d\not=d'$, and for each sort $X\in\B$ there is the set
      \[\approx^{p,p'}_X:=\{(\subseteq_j(\FD_p(y)),\subseteq_{j'}(\FD_{p'}(y)))
                          \mid y\in\FD_i(X)\} 
      \]
      of pairs (primary identifications) in $\coprod_{i\in Min(\I)}\FD_i$.\footnote{
      Recall the previous definition of $\FD_p$ in Sect.~\ref{sec:cyclic} as the compositions of all $\FD_d$ with $d$ an arrow of path $p$
} 
Each pair is represented by a mapping path (called a {\em primary} mapping path) connecting the pair's components\footnote{It can be shown inductively over the path length that each pair can even be represented by a path {\em with exactly one branching position}.}. By $\approx_X$ we denote the union of all those sets $\approx^{p,p'}_X$ for sort $X$. This results in a family $\approx\,=(\approx_X)_{X\in\B}$ of binary relations in $\coprod_{i\in Min(\I)}\FD_i$.
  \item \label{enum:spec-colim3} Construct the smallest congruence $\cong\,=(\cong_X)_{X\in\B}$ in $\coprod_{i\in Min(\I)}\FD_i$ which comprises $\approx$ by enlargement with transitive (i.e.\ concatenation of primary mapping paths) and reflexive (empty mapping paths) closure\footnote{$\approx$ is already symmetric by definition and the transitive closure preserves symmetry. It is then easy to see that compatibility with operation symbols holds.}.
  \item Construct the colimit object as the sortwise quotient $(\coprod_{i\in Min(\I)}\FD_i)/\cong$ and get, in such a way, also the canonical morphisms $[\;]_\cong:\coprod_{i\in Min(\I)}\FD_i\to (\coprod_{i\in Min(\I)}\FD_i)/\cong$.
  \item\label{enum:spec-colim5} The colimiting cocone of diagram $\FD$ is given by
\begin{equation}\label{eqn:special-colimit-cocone}
\FD\stackrel{\kappa}{\Rightarrow}(\coprod_{i\in Min(\I)}\FD_i)/\cong
\end{equation}
where $\kappa_i := [\;]_{\cong}\circ \subseteq_i$ for all minimal indices $i\in Min(\I)$ and $\kappa_i:=\kappa_j\circ\FD_p$ for all other indices $i\in\I_0\setminus Min(\I)$, where $p=(i\to\ldots\to j)$ is an edge sequence from $i$ to $j\in Min(\I)$.
\end{enumerate}
 A detailed proof for the validity of \myref{eqn:special-colimit-cocone} is given in \cite{kw17}. Note, that the definition of $\kappa_i$ is independent of the choice of $p$ since we have, by construction, $\kappa_j\circ\FD_p=\kappa_{j'}\circ\FD_{p'}$ for all branches $p=(i\to\ldots\to j)$, $p'=(i\to\ldots\to j')$ in $\I$ with a common source.
\par
The main advantage of this construction is that mapping paths are now computed traversing branching components only. These components, however, are often just tiny "connectors" as in Fig.\ref{fig:example1}.  Note, that all the components $\FD_i$ with $i\in Min(\I)$ not being the target of any branch are not affected by the quotient construction, i.e.\ congruence classes $[z]_\cong$ are singletons for all $z\in \FD_i(X)$. In Example \ref{ex:shape-graph} this is the case for index $3$. Let $Af(\I)$ denote the set of all \underline{af}fected minimal indices. We could then be even more specific and construct the colimit object as
\begin{equation}\label{equ:affected-components}
(\coprod_{i\in Min(\I)\setminus Af(\I)}\FD_i)+(\coprod_{i\in Af(\I)}\FD_i)/\cong
\end{equation}

\subsection{Efficient Checking of the Van Kampen Property}\label{sec:VK-check-efficient}
Based on \myref{equ:affected-components} we can conclude, independent of Corollary \ref{cor:main}, that a diagram $\FD:\I\to\pre{\B}$ is VK if $\I$, in addition of being finite and having no directed cycles, does not have branching indices either. In this case, $Af(\I)$ is empty and the colimit of the diagram is simply given by the coproduct of all minimal components, thus the VK property of the diagram is ensured by the VK property of coproducts (extensivity) and pullback composition. In the presence of branching, however, we do not have VK for free. We have to check one of the conditions in Cor.\ref{cor:main}. 
\par 
The specialized colimit construction \myref{equ:affected-components} suggests another practical relevant possibility to reduce our effort for checking VK in case $\I$ has no directed cycles. In applications that deal with nets of software components (e.g. multimodels), there is usually only one type of relation between the components: Relations either specify sameness of model elements, versions of one model element in evolving environments, or elements to be preserved when applying transformation rules \cite{Ehrig2006-foagt}. Thus, rarely will it be the case that there are two or more morphisms in the same direction between two given components. An even weaker and also reasonable claim for two different relations is that they don't interfere in common codomains, thus the following definition is not too restrictive: 
\begin{defi}[Image-Disjointness]\label{def:image-disj}
A diagram $\FD:\I\to\pre{\B}$ is called {\em image disjoint}, if for each pair of different branches $p=(i\to\ldots\to j)$, $p'=(i\to\ldots\to j')$ in $\I$ starting in the same branching index $i$ and all elements $y\in\FD_i(X)$, $X\in\B$ we have $\FD_{p}(y)\not= \FD_{p'}(y)$.
\end{defi}
\begin{fact}\label{fact5}
If $\FD$ is not image-disjoint, the colimt  $\FD\Rightarrow\Delta S$ does not have the Van Kampen property. 
\end{fact}
\begin{proof}
The two different branches $p,p'$ and $y\in \FD_i(X)$ with $\FD_{p}(y)= \FD_{p'}(y)$ yield two different mapping paths from $y$ to $\FD_{p}(y)$. Thus the result follows from Cor.\ref{cor:main}. 
\end{proof}
If there are no undirected cycles in $\I$, then we have image disjointness for free, because we always assume that all components $\FD_i$ are pairwise disjoint. If there are undirected cycles in $\I$, as in the case of coequalizers, for example, it can happen that $j=j'$. Thus we have to test, first, for image disjointness before the ''different paths criterion for paths connecting affected minimal components'' below can be applied. Note, that image disjointness implies that we have $\FD_{p}(y)\not= \FD_{p'}(y)$ for all $y\in\FD_i(X)$, $X\in\B$ not only for branches but for arbitrary pairs of paths $p=(i\to\ldots\to j)$, $p'=(i\to\ldots\to j')$ starting in a common branching index $i\in Br(\I)$ but not necessarily ending at a minimal index.
\begin{thm}[Different Paths Connecting Affected Minimal Components]\label{theo:main_minimal}
Let $\G = \pre{\B}$ and $\FD:\I\to\G$ be a diagram with $\I$ a finite directed multigraph without directed cycles. Let
\[\FD\stackrel{\kappa}{\Rightarrow} \Delta S\]
be a colimiting cocone with image-disjoint $\FD$. 
The following are equivalent: 
\[\begin{array}{cll}
(1) & \mbox{The cocone has the Van Kampen property} & \\
(2) & \forall X\in \B, i,j\in Af(\I), z\in\FD_i(X), z'\in\FD_j(X): & 
\mbox{There are no two different proper } \\
&& \mbox{paths from $z$ to $z'$}\\
(3) & \forall X\in \B, i,j\in Af(\I), z\in\FD_i(X), z'\in\FD_j(X): & 
\mbox{There are no two different inner-cycle } \\
&& \mbox{free proper paths from $z$ to $z'$} \hfill  
\end{array}
\]
\end{thm}
\begin{proof}
The proof of Theorem \ref{theo:main_minimal} is rather elementary and is given in \cite{kw17}.
\end{proof}
According to this theorem and according to \myref{equ:affected-components}, it is only necessary for an algorithm to iterate over branching components and compute paths into potentially affected minimal components. For instance, one has to consider only the small components $\FD_{12}, \FD_{13}, \FD_{23}$ in Fig.\ref{fig:example1}. Again the implementation is independent of whether it ignores inner-cycle free paths or not. It is, however, not guaranteed to find disjoint paths, if VK is violated.
\par
The absence of directed cycles and the presence of image-disjointness are reasonable requirements for many practical use-cases. But circumstances can often be further narrowed. In the rest of this section we consider some other possibly satisfied properties and corresponding alternative checking methods which can simplify VK verification. 
\paragraph{{\em Monomorphisms}:} In some practical cases, the morphisms of $\FD$ specify relations between components $\FD_i$ and $\FD_j$ such that an element in $\FD_i$ is related to at most one element of $\FD_j$. In this case all the morphisms $\FD_d$, with $d$ an edge in $\I$, are monomorphisms. In such a diagram any mapping path $P$ is completely determined by $y_0$ (or $y_n$) and the corresponding sequence $[\delta_0,...,\delta_{n-1}]$ of edges and opposed edges in $\I$. Due to Theorem \ref{theo:main_minimal}, the diagram may be not VK only if there are two different sequences of edges and opposed edges between two distinct affected indices in $\I$. As long as there are no undirected cycles in $\I$ this can not happen, thus the diagram is VK, if there are no undirected cycles in $\I$.
\par
If there are undirected cycles in $\I$ it is surely not enough to require image-disjointness as defined in Def.\ref{def:image-disj}, see also Fig.\ref{fig:example1}. Instead, we can ensure VK by the stronger requirement that all the undirected cycles in $\I$ are broken in $\FD$: An undirected cycle of edges in $\I$ is {\bf broken in} $\FD$ if for one of the situations  
\[ \ldots \stackrel{d_{n-1}}{\longrightarrow}\cdots\stackrel{d_0}{\longrightarrow}
   \stackrel{d'_0}{\longleftarrow}\cdots\stackrel{d'_{m-1}}{\longleftarrow}\ldots
\]
in the edge sequence with $1\leq n,m$ the morphisms $\FD_{d_0}\circ\ldots\circ\FD_{d_{n-1}}$ and $\FD_{d'_0}\circ\ldots\circ\FD_{d'_{m-1}}$ are image disjoint.

In the example in Figure \ref{fig:example1} this condition is not satisfied. In the example "parametrized specification with import", however, it is quite natural that $\FD_{ib}$ and $\FD_r$ are image disjoint since the "imported component" $\FD_I$ is not part of the "parameter component" $\FD_p$.

\subsection{Decision Diagram and VK Verification Algorithm}\label{sec:dec-diagram}
\begin{figure}[t]
\begin{small}
  \begin{center}
  $\xymatrix{
\ar[r]^(0.2){\txt{start}}
  &  \txtnode{18mm}{Directed cycles in $\I$?}
    \ar[d]^(.5){\txt{yes}} \ar[r]^{\txt{no}}
  & \txtnode{18mm}{Branching in $\I$?}\ar@/^2pc/[ddr]^(.6){\txt{no}} \ar[r]^{\txt{yes}}  
  & \txtnode{18mm}{$\FD$ image-disjoint?}\ar[r]^(.5){\txt{yes}} \ar@/_2pc/[ddl]|(0.32)\hole_(.6){\txt{no}}
  & \txtnode{20mm}{Only monic's in $\FD$?}\ar[d]^(.5){\txt{yes}} \ar@/^4pc/[dd]^{\txt{no}}
   \\
& \txtnode{20mm}{One $\FD_i(X)$ on the cycle finite?}\ar[dr]^(.5){\txt{yes}} \ar[d]^{\txt{no}}
  &  
  & 
  & \txtnode{22mm}{All undirected cycles broken in $\I$?}\ar[dl]^{\txt{yes}} \ar[d]^{\txt{no}}
\\
&  \txtnode{15mm}{Apply Cor.\ref{cor:main}} & \txtnode{20mm}{\bf not VK} & 
\txtnode{12mm}{\bf VK}& \txtnode{15mm}{Apply Thm.\ref{theo:main_minimal}} 
  }$\\
  \end{center}
\end{small}
  \caption{Decision diagram}
  \label{fig:decision-diagram}
\end{figure}
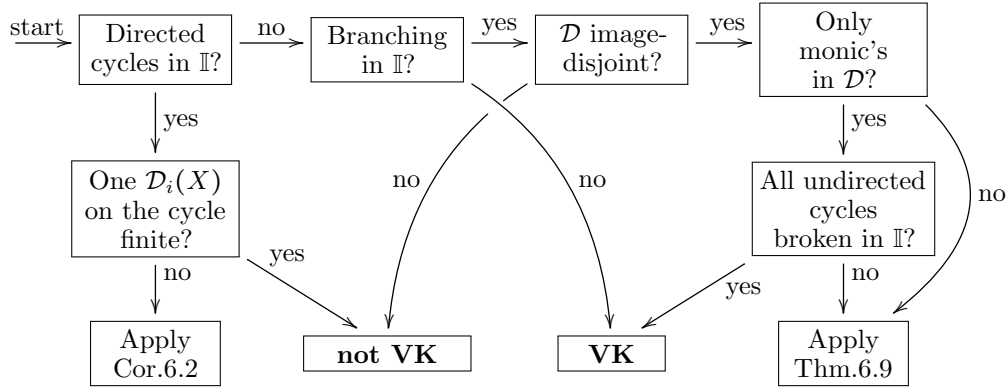
In practical cases, as outlined e.g.\ in Sect.\ref{sec:appl}, colimit computation is obligatory. Verification of the Van Kampen property must follow, if we want to verify compositionality. It would thus be a nice side effect to have a possibility to check VK simultanously with colimit computation such that there is no increase in time complexity! We will now shortly discuss, that with the results gained so far, this is indeed possible.
\par
For this, let's summarize the outcome of the previous sections as a decision algorithm, see Fig.\ref{fig:decision-diagram}. With the exception of rare cases in which there is a directed cycle in $\I$ whose component carrier sets are all infinite, it is possible to easily reach an early decision, if either there are directed cycles in $\I$ or if there is no branching in $\I$. This analysis is restricted to the small graph $\I$. In the presence of branching, the natural next step is to check violation of image-disjointness in $\FD$ to immediately deduce violation of VK (Fact \ref{fact5}). Image-disjointness can immediately be confirmed, if all branches diverge. Otherwise, the mapping behavior along branches has to be investigated, which may be more costly.   
\par 
Hence the {\em combined algorithm for colimit computation and Van Kampen verification} comprises the following steps:
\begin{enumerate}
  \item Preprocessing of the data shows whether we are on a decision route in Fig.\ref{fig:decision-diagram} on which Thm.\ref{theo:main_minimal} will be applied. In this case there are {\em no directed cycles in} $\I$ and $\FD$ is {\em image-disjoint}.    
  \item $vk:=true$;
  \item $\cong_X:=\{(z,z) \mid z\in \FD_j(X)$, $j\in Min(\I)\}$ for all $X\in\B$;
  \item For each branching component $\FD_i$, each $X\in\B$ and for each $y\in \FD_i(X)$, do: 
  \begin{enumerate}
    \item Add images $(z,z')$ to $\cong_X$ according to primary identifications in step \ref{enum:spec-colim2} in the specialized colimit computation. 
    \item Keep $\cong_X$ transitive by adding all arising transitive pairs from the last enhancement.   
    \item Whenever in the two previous steps a pair $(z,z')$ is added for the second time, $vk:=false$ (cf.\ Theorem \ref{theo:main_minimal}).
  \end{enumerate} 
  \item Compute colimit cocone $\kappa$ as in \myref{eqn:special-colimit-cocone} using the family  $\cong=(\cong_X)_{X\in\B}$.
  \item Return $(\kappa,vk)$ 
\end{enumerate}





\section{Conclusion and Future Work}\label{sec:conc}
In general, arbitrary diagrams in arbitrary categories are not VK. Even if we restrict to presheaf topoi, many diagrams are not VK. In the paper we presented a feasible condition (Thm. \ref{theo:main}) to check if a diagram in a presheaf topos is VK or not.
\par
As suggested by the example in Sect.\ref{sec:appl}, modelers may well work with a non-VK-diagram (of software models), if they have a common understanding of the used natural transformation $\tau:\FE\Rightarrow \FD$, i.e., if they know how to avoid "twisting anomalies" as shown in the example. Hence, the natural next step will be to look for feasible conditions that a given $\tau:\FE\Rightarrow \FD$ is in the image of $\kappa^\ast$, even if the diagram is not VK. We may allow non-uniqueness of  mapping paths in diagrams of models, but then paths in the diagram of instances have to be exact copies of them, i.e., path liftings from models to instances must behave like discrete fibrations. It is worth to underline that the instances we get from a given "indexed semantics" via a corresponding variant of the Grothendieck construction \cite{WD2007} are always contained in the image of $\kappa^\ast$ up to isomorphism.
\par
An interesting research direction arises from counter-examples in Sect. \ref{sec:counterex}. It seems to be easy to find categories, where necessity of the Van Kampen property is violated although path uniqueness holds. Although being artificial and practically not relevant, the example showing violation of path uniqueness despite validity of VK is interesting: there do not seem to be other substantially different examples of this type. Is it possible to have violation of path-uniqueness and still validity of the Van Kampen property in more practical examples? We conjecture that the implication ''VK $\Rightarrow$ Path-Uniqueness'' is very natural and holds in a wider variety of  (set-based) categories. 
\par 
For topologists being familiar with homotopy theory \cite{MayAlgTop}, violation of path uniqueness strongly resembles (continuous) paths that can not be contracted to a point. Hence, an interesting further research direction is to investigate, how path lifting (e.g.\ of covering morphisms) is connected with our investigations. Moreover, discrete unique path lifting is discussed in the theory of (split) fibrations \cite{St1999}. The ultimate goal, however, is to find a categorical counterpart for the path-uniqueness criterion (Theorem \ref{theo:main}), which states a necessary and sufficient condition for validity of the Van Kampen property in more general categories. Is such a condition significantly different from the bilimit condition mentioned in the introduction and the universal property in \cite{VKAsBicolimits} and can we benefit from results of higher order category theory, e.g.\ \cite{Lurie}?


\bibliographystyle{plainurl}

 
\end{document}